\documentclass[twoside,12pt,reqno]{amsart}

\usepackage{amssymb}
\usepackage{fullpage}
\usepackage{latexsym}
\usepackage{mathrsfs}
\usepackage{enumitem}
\usepackage{mathdots}
\usepackage{euscript}

\usepackage[pdftex,colorlinks,backref,pagebackref,hypertexnames=false,
           linkcolor=blue,urlcolor=blue,citecolor=blue]{hyperref}

\newtheorem{theorem}{Theorem}[section]
\newtheorem{prop}[theorem]{Proposition}
\newtheorem{lemma}[theorem]{Lemma}
\newtheorem{cor}[theorem]{Corollary}
\newtheorem{ksprop}{KS-property}
\newtheorem{obs}{Observation}

\theoremstyle{definition}
\newtheorem{exa}[theorem]{Example}
\newtheorem*{dfn}{Definition}

\theoremstyle{remark}
\newtheorem{rmk}[theorem]{Remark}

\numberwithin{equation}{section}

\input{xy}
\xyoption{all}

\def\bar{\overline}
\def\sub{\subseteq}

\def\onto{\twoheadrightarrow}

\newcommand{\isoto}{\overset{\sim}{\longrightarrow}}

\def\b{{\mathfrak b}}
\def\g{{\mathfrak g}}
\def\h{{\mathfrak h}}

\def\l{{\mathfrak l}}
\def\m{{\mathfrak m}}
\def\p{{\mathfrak p}}
\def\t{{\mathfrak t}}
\def\u{{\mathfrak u}}
\def\z{{\mathfrak z}}

\def\gl{{\mathfrak{gl}}}
\def\sl{{\mathfrak{sl}}}
\def\so{{\mathfrak{so}}}
\def\sp{{\mathfrak{sp}}}

\def\GL{{\mathrm{GL}}}
\def\SL{\mathrm{SL}}
\def\SO{\mathrm{SO}}
\def\Sp{\mathrm{Sp}}

\def\ab{\mathrm{ab}}
\def\op{\mathrm{op}}
\def\reg{{\mathrm{reg}}}
\def\tr{\mathrm{tr}}

\def\C{{\mathbb C}}
\def\O{{\mathbb O}}
\def\Z{{\mathbb Z}}

\def\cD{{\mathcal D}}
\def\cE{{\mathcal E}}

\def\cR{{\mathcal R}}
\def\cS{{\mathcal S}}

\def\cVA{{\mathcal{VA}}}

\def\sA{{\mathsf{A}}}
\def\sB{{\mathsf{B}}}
\def\sC{{\mathsf{C}}}
\def\sD{{\mathsf{D}}}

\def\sS{\mathscr{S}}

\def\ad{\operatorname{ad}}
\def\ch{\operatorname{ch}}
\def\Ann{\operatorname{Ann}}
\def\Comp{\operatorname{Comp}}
\def\End{\operatorname{End}}
\def\gr{\operatorname{gr}}
\def\Ind{\operatorname{Ind}}
\def\Lie{\operatorname{Lie}}
\def\lmod{\operatorname{-mod}}

\def\Prim{\operatorname{Prim}}
\def\rig{\operatorname{rig}}
\def\Spec{\operatorname{Spec}}
\def\Wh{\operatorname{Wh}}

\def\bi{\text{\boldmath$i$}}
\def\bj{\text{\boldmath$j$}}

\newcommand{\arxiv}[1]{{\tt arXiv:#1}}

\title{On induced completely prime primitive ideals in enveloping algebras of classical Lie algebras}
\author{Simon M.~Goodwin, Lewis Topley and Matthew Westaway}

\address{School of Mathematics, University of Birmingham, Birmingham, B15 2TT, UK}	
\email{S.M.Goodwin@bham.ac.uk}
	
\address{Department of Mathematical Sciences, University of Bath, Bath, BA2 7AY, UK}	
\email{lt803@bath.ac.uk, mw2915@bath.ac.uk}
	
	\thanks{Mathematics Subject Classification (2000 revision). Primary 17B45, 17B10. Secondary 17B08.}

\begin{document}

\begin{abstract}
	A distinguished family of
	completely prime primitive ideals in the universal enveloping algebra of a reductive Lie algebra $\g$ over $\C$
	are those ideals constructed from one-dimensional representations of finite $W$-algebras.  We
	refer to these ideals as {\em Losev--Premet ideals}. For $\g$ simple of classical type, we prove that for
	a Losev--Premet ideal $I$ in $U(\g)$,
	there exists a Losev--Premet ideal $I_0$ for a certain
	Levi subalgebra $\g_0$ of $\g$ such that the associated variety of $I_0$ is the closure of a rigid nilpotent orbit in $\g_0$ and
	$I$ is obtained from $I_0$ by parabolic induction; in turn, this gives a classification of rigid Losev-Premet ideals in $U(\g)$. This is deduced from the corresponding statement about one-dimensional representations of finite
	$W$-algebras.
\end{abstract}

\maketitle

\section{Introduction}
Let $G$ be a connected reductive algebraic group over $\C$ and let $\g = \Lie G$. The study of
the spectrum $\Prim U(\g)$ of primitive ideals in the universal enveloping algebra $U(\g)$ of $\g$ is
one of the most venerable topics in representation theory.
These primitive ideals serve as a tractable approximation to the irreducible representations of $U(\g)$
and have generated an enormous amount of interest from leading mathematicians.  Amongst the primitive
ideals of $U(\g)$, an
especially important role is played by the spectrum $\Prim^1 U(\g)$ of completely prime primitive ideals. This importance is largely due to the close relationship between these ideals and the unitary representations of Lie groups (see \cite{LMM} and the references therein).

The classification of $\Prim U(\g)$ as a set was established in the early 1980s following important
contributions from Barbasch--Vogan, Duflo, Joseph and others; we refer to \cite{JoCl} and
\cite{JoICM} for an overview with references.
Despite the classification of $\Prim U(\g)$ being established 40 years ago, the classification problem
for $\Prim^1 U(\g)$ has proved to be very challenging and remains open.  A classification is known
in the case $G$ is of type $\sA$, where it was established by M{\oe}glin in \cite{Moe}.  There is also
work dealing with some small rank cases, for instance the work of Borho in \cite{BoB2} for $G$
of type $\sB_2$, and there have been many further important
contributions to the general classification problem. We refer to \cite[\S1.2]{PT} for some discussion and references.

A general paradigm in classification problems in Lie theoretic representation theory is
to use a process of parabolic induction. For the case of $\Prim^1 U(\g)$,
parabolic induction of ideals, as recalled in \S\ref{ss:primitive}, can be used.
Of importance for this approach is the theorem of Conze in \cite{Co}, which states that
the induced ideal of a completely prime primitive ideal is itself a completely prime primitive ideal.
Thus to classify completely prime primitive ideals a key problem is to determine the
so-called {\em rigid} completely prime primitive ideals, namely those that cannot be non-trivially obtained
by parabolic induction.  The main result of this paper
solves this problem, in the case $G$ is of classical type,
for an important family of completely prime primitive ideals
that arise naturally from the theory of finite $W$-algebras.
Before stating this result in Theorem~\ref{T:maintheorem}, we need to recall
some more background.

Following the classification of $\Prim U(\g)$, an important problem was to
determine the associated variety $\cVA(I)$ of a primitive ideal $I$; the definition
of $\cVA(I)$ is recalled in \S\ref{ss:primitive}. This problem was
made precise by Joseph in \cite{JoAV}: 
by identifying $\g$ with its dual via a nondegenerate invariant symmetric bilinear form,
$\cVA(I)$ is viewed as a closed subvariety of $\g$, and the main result of {\it op. cit.} states that $\cVA(I)$ is the closure $\bar \O_e$ of the $G$-orbit $\O_e$
of some nilpotent element $e \in \g$.  This leads to the decomposition $\Prim U(\g)  = \bigsqcup \Prim_{\O} U(\g)$,
where the (finite, disjoint) union is taken over the
nilpotent $G$-orbits $\O$  in $\g$ and $\Prim_{\O} U(\g) : = \{I \in \Prim U(\g) \mid \cVA(I) = \bar \O \}$.
Consequently, problems regarding $\Prim U(\g)$ can broken down to problems
about $\Prim_{\O} U(\g)$, and those about $\Prim^1 U(\g)$  can be reduced
to $\Prim^1_{\O} U(\g) := \Prim^1 U(\g) \cap \Prim_{\O} U(\g)$.

The introduction of finite $W$-algebras to the mathematical
literature by Premet in \cite{PrST} led to a resurgence
in interest in the primitive ideals of $U(\g)$.  In \S\ref{S:Walg}, we recall
some background on the finite $W$-algebra $U(\g,e)$ associated to a nilpotent element $e \in \g$.
In particular, we recall the close
relationship between $\Prim_{\O_e} U(\g)$ and the finite dimensional irreducible representations of the finite $W$-algebra
$U(\g,e)$.
This gives an approach to $\Prim_{\O_e} U(\g)$ via the representation theory of $U(\g,e)$, which
has led to spectacular progress.  For instance, it leads to a resolution of the problem discussed in \cite{BJ}
that $\Prim U(\g)$ can be described as a countable union of varieties; the work of Losev in \cite{LoGR}
made a breakthrough in the
programme proposed by Joseph in \cite{JoGR} to determine Goldie rank polynomials; and the longstanding
problem that $\Prim^1_{\O_e} U(\g) \ne \varnothing$ has been
resolved, with the proof completed by Premet in \cite{PrMF}.

The theory of finite $W$-algebras picks out a distinguished class of completely prime
primitive ideals, namely those corresponding to one-dimensional representations of finite $W$-algebras.  We
name these ideals {\em Losev--Premet} ideals, reflecting the major progress on their theory made in the work
of both Losev and Premet.
Thanks to \cite[Theorem~1.2.2]{LoQS} every
Losev--Premet ideal coming from a one-dimensional representation of $U(\g,e)$ lies in $\Prim^1_{\O_e} U(\g)$.
It is now well-established that the Losev--Premet ideals form an important
family in $\Prim^1_{\O_e} U(\g)$.  For instance, it is proved in \cite[Corollary 1.2]{LoQuant} that
for primitive ideals with integral central
character, the Losev--Premet ideals give all completely prime primitive ideals (with a possible exception
for $G$ of type ${\sf E}_8$); it is known in type {\sf A} that this statement holds without the restriction on central character (see the remarks following \cite[Theorem~B]{Pr11}); it is proved in \cite[Theorem 1.2]{To} that, for $G$ of classical type, the Losev--Premet ideals
are precisely the ideals in the image of Losev's orbit map introduced in \cite{LoOM}; and
it is expected that the
left and right annihilators of unitary Harish-Chandra bimodules for $U(\g)$ are Losev--Premet ideals,
see \cite[Conjecture~6.3.1]{LMM}.  We note that it is in general false that every
ideal in $\Prim^1 U(\g)$ is a Losev--Premet ideal and refer to the work of Losev--Panin in \cite{LP21}
for an important recent development regarding this defect.

We now state our main theorem, which constitutes major progress in our understanding of Losev-Premet ideals and, as explained after its statement,
completes the classification of rigid Losev--Premet ideals for $G$ of classical type.
In the statement we refer to Lusztig--Spaltenstein induction of nilpotent orbits, which is recapped in
\S\ref{ss:LSind}.

\begin{theorem} \label{T:maintheorem}
	Let $G$ be a simple algebraic group over $\C$ of classical type and let
	$\O$ be a nilpotent orbit in $\g = \Lie G$.
	Let $I \sub U(\g)$ be a Losev--Premet ideal with $\cVA(I) = \bar{\O}$.
	Then there exists a Levi subalgebra $\g_0$ of $\g$, a parabolic subalgebra $\p$ of $\g$ with Levi factor $\g_0$, a rigid nilpotent
	orbit $\O_{0}$ in $\g_0$, and a Losev--Premet ideal $I_0 \sub U(\g_0)$ such that
	$\O$ is obtained from $\O_{0}$
	by Lusztig--Spaltenstein induction, $\cVA(I_0) = \bar \O_{0}$,
	and $I$ is
	induced from $(\p, I_0)$.
\end{theorem}

This theorem implies that the rigid Losev--Premet ideals, for $G$ of classical type, are precisely those
whose associated variety is the closure of a rigid nilpotent orbit.
By \cite[Theorem~1]{PT} it is known that, for $G$ simple of classical type
and $\O$ a rigid nilpotent orbit in $\g$, there is a unique Losev--Premet ideal $I$ with $\cVA(I) = \bar \O$.
Moreover, the highest weight of a simple highest weight $U(\g)$-module with annihilator $I$ can be
determined using \cite[Proposition 8.2.3]{LMM}.  Thus Theorem~\ref{T:maintheorem} completes
the classification of rigid Losev--Premet ideals for $G$ of classical type.

Theorem~\ref{T:maintheorem} can be viewed as an extension
of \cite[Theorem\ 5]{PT} and our methods are a novel development of the methods there,
whilst also making use of recent results from \cite{To}.
The proof deduces the theorem from the corresponding statement for finite $W$-algebras.
In Theorem~\ref{T:surj} we prove that, with the set up in Theorem~\ref{T:maintheorem},
any one-dimensional $U(\g,e)$-module can be obtained
from a one-dimensional $U(\g_0,e_0)$-module
via the parabolic induction functor introduced by Losev in \cite{LoPI}.

We give now
a brief overview of the strategy of the proof of Theorem~\ref{T:surj}. First we need to establish some notation. The variety of one-dimensional $U(\g,e)$-modules is denoted $\cE(\g,e)$, see \eqref{e:Ege} for a formal definition.  We say that
$(\g_0,\O_{0})$ is a rigid induction datum for $(\g,\O)$ if $\O$ is induced from $\O_{0}$
and $\O_0$ is rigid in $\g_0$. Parabolic induction for finite $W$-algebras gives a finite morphism $\cE(\g_0,e_0) \to \cE(\g,e)$
for any rigid induction datum $(\g_0,\O_{e_0})$ for $(\g,\O_e)$.  In Proposition
\ref{P:dimpresbij} it is shown that the
multiset of dimensions of irreducible components of the disjoint union $\bigsqcup \cE(\g_i,\O_{e_i})$
over the ($G$-orbits of) rigid induction data $(\g_i,\O_{e_i})$ for $(\g,\O_e)$ is
equal to the multiset of dimensions of irreducible components
of $\cE(\g,e)$.  Combining this with Proposition~\ref{P:Int}, which shows that distinct
irreducible components of $\bigsqcup \cE(\g_i,e_i)$ are mapped by the parabolic induction morphism
to distinct irreducible components of $\cE(\g,e)$, we can deduce Theorem~\ref{T:surj}.  An important result
along the way is Theorem~\ref{T:nonconj}, which shows that for distinct rigid induction data
$(\g_1,\O_{e_1})$ and $(\g_2,\O_{e_2})$ for $(\g,\O_e)$ we have $\g_1 \not\sub \g_2$ and $\g_2 \not\sub \g_1$;
we refer also to Remark~\ref{R:nonconj} for the extension of this to any reductive $G$.

We remark that the computer calculations reported in \cite{BG} show that the analogue of
Theorem~\ref{T:maintheorem}  for $G$ of exceptional type is not true.  Those calculations show that there
are Losev--Premet ideals in $U(\g)$ for $G$ of type $F_4$ or $E_6$, with associated variety a non-rigid
nilpotent orbit, that are not parabolically induced.  It is not known
whether there are such ideals for $G$ of type $E_7$ or $E_8$, though this seems quite plausible.

To end this introduction we draw attention to the remarkable recent work of  Losev in \cite{LoOM} and
Losev--Mason-Brown--Matvieievskyi in
\cite{LMM}.  An orbit method for Lie algebras was introduced by Losev in \cite[\S5]{LoOM},
which gives a natural embedding from the set of coadjoint orbits of $G$ to $\Prim U(\g)$.
It was conjectured that the ideals obtained in this way for classical Lie algebras are precisely the Losev--Premet ideals,
and as noted above this has now been proved by the second author in \cite{To}.  In \cite[Theorem 7.8.1]{LMM}, Losev's orbit method is extended to define
an injection from isomorphism classes of $G$-equivariant covers of coadjoint orbits to isomorphism classes of
$G$-equivariant filtered quantizations of affinizations of $G$-equivariant covers of nilpotent orbits.
Each of these quantizations gives rise to a primitive ideal, completing Vogan's orbit method program, which in turn has important applications
to the theory of unipotent representations of Lie groups.  Furthermore, their methods give another approach
to understanding Losev--Premet ideals via birationally rigid induction data, see for instance \cite[Remark 8.1.2]{LMM}.

This paper is organised as follows. In \S\ref{S:redgrps} we recall preliminaries on reductive groups, Levi subgroups, nilpotent orbits, Lusztig-Spaltenstein induction, sheets and primitive ideals; we also use this section to introduce notation used throughout the paper. Then, in \S\ref{S:LSinduction} we prove Theorem~\ref{T:nonconj} by means of a combinatorial argument involving the Kempken-Spaltenstein algorithm. In \S\ref{S:Walg} we introduce finite $W$-algebras and establish the framework needed to state Theorem~\ref{T:surj}; the paper concludes in \S\ref{S:FWACLA} with the proof of this theorem.

\subsection*{Acknowledgements}
We would like to thanks Sasha Premet for a useful comment on the first version of this paper. We would also like to thank the referees for their careful reading and helpful comments on the paper. All three authors were supported during this research by the EPSRC grant EP/R018952/1. The
research of the second author is supported by a UKRI Future Leaders Fellowship, grant numbers MR/S032657/1, MR/S032657/2, MR/S032657/3.
The third author gratefully acknowledges funding from the Royal Commission for the Exhibition of 1851.


\section{Reductive groups, Levi subgroups, nilpotent orbits, Lusztig--Spaltenstein induction, sheets and primitive ideals}
\label{S:redgrps}

\subsection{Reductive groups} \label{ss:notn}

Throughout this paper $G$ is a connected reductive algebraic group over $\C$.
We often specialize to the case where $G$ is simple of classical type, by which we
mean $G$ is one of $\SL_n(\C)$, $\SO_N(\C)$, or $\Sp_{2n}(\C)$,
where $n,N \in \Z_{>0}$.
The results of our paper do not depend on the isogeny class of $G$, which justifies
the restriction to these choices of $G$ for types $\sA$, $\sB$, $\sC$ and $\sD$.
As we work entirely over $\C$ in this paper we allow
ourselves to abbreviate notation and just write $\SL_n$, $\SO_N$, or $\Sp_{2n}$ for
these groups;
we write $\sl_n$, $\so_N$ and $\sp_{2n}$ for their Lie algebras.
Explicitly we take $\SO_N = \{x \in \SL_N \mid x^{\tr}J_Nx = J_N\}$, where 
$$
J_N = \left( \begin{array}{ccc}  &  & 1 \\ & \iddots & \\ 1 & & \end{array} \right)
$$ 
is the $N \times N$ matrix with 
$(i,j)$ entry equal to $\delta_{i,N+1-j}$; and we take $\Sp_{2n} = \{x \in \SL_{2n} \mid x^{\tr}\tilde J_{2n}x = \tilde J_{2n}\}$
where $\tilde J_{2n}$ has block form 
$$
\tilde J_{2n} = \left( \begin{array}{cc} 0 & J_n \\ - J_n & 0 \end{array} \right).
$$ Note that for these choices of $J_N$ and $\tilde J_{2n}$, the resulting algebraic groups have the property that the subgroup of diagonal matrices forms a maximal torus and the subgroup of upper triangular matrices forms a Borel subgroup.

For a closed subgroup $H$ of $G$ we write $\h = \Lie H$ for the Lie algebra of $H$ and $\z(\h)$
for the centre of $\h$.
We let $(\cdot\,,\cdot)$ be a nondegenerate $G$-invariant symmetric bilinear form on $\g$, which allows
us to identify $\g \cong \g^*$ as $G$-modules. We have that $G$ acts on $\g$ via the adjoint action.  Let $x \in \g$, let $X \sub \g$ and let $H$ be a subgroup of $G$.
We write $H \cdot x$ for the $H$-orbit of $x$, though often use the notation $\O_x$ for $G \cdot x$.  The centralizer of $x$ in $H$ is denoted $H^x$ and
the centralizer of $x$ in $\h$ is denoted $\h^x$.
We write $H \cdot X$ for the $H$-saturation of
$X$, i.e.\ $H \cdot X = \bigcup_{x \in X} H \cdot x$.
The closure of $X$ with respect to the Zariski topology on $\g$ is denoted $\bar X$.

We fix a maximal torus $T$ of $G$ and a Borel subgroup $B$ of $G$ with $T \sub B$, and let $U$
denote the unipotent radical of $B$; as noted above, we may pick $T$ to consist of diagonal matrices and $B$ to consist of upper triangular matrices. We denote by $\t$, $\b$ and $\u$ the respective Lie algebras of these subgroups.
Let $\Phi$ be the root system of $G$ with respect to $T$, let $\Phi^+$ be the set
of positive roots determined by $B$ and let $\Pi$ be the corresponding set of simple roots.
We write $h_\alpha \in \t$ for the coroot corresponding to $\alpha \in \Phi$.

The Weyl group of $G$ with respect to $T$ is denoted $W$.
We choose
$\rho \in \t^*$ such that $\rho(h_\alpha)=1$ for all $\alpha\in \Pi$; for instance we could
take $\rho$ to be the half sum of the roots in $\Phi^+$.
The {\em dot action} of $W$ on $\t^*$ is
defined by $w \bullet \lambda := w(\lambda+\rho)-\rho$, and we note that this does not depend
on the choice of $\rho$.

For $\alpha \in \Phi$, we write $\g_\alpha \sub \g$ for the corresponding root subspace.
Given a closed subgroup $H$ of $G$ containing $T$, we define $\Phi(\h)$ to be the subset of $\Phi$ such that
$\h = \t \oplus \bigoplus_{\alpha \in \Phi(\h)} \g_\alpha$.  Given a subset $\Gamma$ of $\Pi$, we define
$\Phi_\Gamma$ to be the root subsystem of $\Phi$ generated by $\Gamma$, more precisely
$\Phi_\Gamma = \{\sum_{\alpha \in \Gamma} a_\alpha \alpha \in \Phi \mid a_\alpha \in \Z\}$.

\subsection{Parabolic subgroups and Levi subgroups} \label{ss:paralevi}

We recall some results about parabolic subgroups and their Levi factors that can be found, for example, in \cite[\S3.8]{CM}.

First we recap that a subgroup $P$
of $G$ is a parabolic subgroup if it contains a conjugate of $B$.
We say that $P$ is a standard parabolic subgroup
if $B \sub P$, so every parabolic subgroup is conjugate to a standard parabolic
subgroup.  For a standard parabolic subgroup $P$ with unipotent radical $U_P$, there is a subset $\Gamma$
of $\Pi$ such that $\Phi(\p) = \Phi_\Gamma \sqcup \Phi(\u_P)$, and we use the notation $P = P_\Gamma$.  For $\Gamma, \Delta \sub \Pi$, we have that
$P_\Gamma$ is conjugate to $P_\Delta$ if and only if $\Gamma = \Delta$. Thus
the conjugacy classes of parabolic subgroups of $G$ are in bijection with subsets of $\Pi$.

A parabolic subgroup $P$ of $G$ has a Levi decomposition $P = G_0 U_P$, where $U_P$ is the unipotent radical of $P$
and $G_0$ is a Levi factor.  By a Levi subgroup of $G$ we mean a Levi factor $G_0$ of some
parabolic subgroup, and we refer to $\g_0$ as a Levi subalgebra.

We recall that Levi subalgebras can equivalently be defined as centralizers of semisimple elements. More precisely,
for any semisimple $s \in \g$, we have that $\g^s$ is a Levi subalgebra of $\g$ and all Levi subalgebras
are of the form $\g^s$ for some semisimple $s \in \g$.  
In particular, this implies that if $\g_1$ and $\g_2$ are Levi subalgebras
of $\g$ with $\g_1 \sub \g_2$, then $\g_1$ is a Levi subalgebra of $\g_2$.

For a standard parabolic subgroup $P$ there is a unique Levi factor $G_0$ containing $T$,
which we refer to as a standard Levi subgroup, and then we refer
to $\g_0$ as a standard Levi subalgebra.  
We note that any Levi subgroup is
conjugate to a standard Levi subgroup.
For a standard Levi subgroup $G_0$, we have that
$\Phi(\g_0) = \Phi_\Gamma$ for some $\Gamma \sub \Pi$, and we use the notation $G_\Gamma$ for $G_0$.
For $\Gamma,\Delta \sub \Pi$
we have $G_\Gamma$ is conjugate to $G_\Delta$ if and only if there exists $w \in W$ such that $w \cdot \Gamma = \Delta$.

We make the following observation for later. Suppose that $G_1$ and $G_2$ are standard Levi subgroups of $G$ with Lie algebras $\g_1$ and $\g_2$ respectively, and that there exists a $G$-conjugate $\tilde \g_1$ of $\g_1$ such that $\tilde \g_1\subseteq \g_2$. Then $\tilde \g_1$ is a Levi subalgebra of $\g_2$, and thus is $G_2$-conjugate (and hence $G$-conjugate) to a standard Levi subalgebra $\widehat{\g}_1$ of $\g_2$. It is straightforward to see that standard Levi subalgebras of $\g_2$ are standard Levi subalgebras of $\g$. We therefore conclude that any inclusion of Levi subalgebras of $\g$ is $G$-conjugate to an inclusion of standard Levi subalgebras of $\g$.

For a Levi subgroup $G_0$, we define
$\z(\g_0)^\reg := \{x \in \z(\g_0) \mid (\g_0)^x = \g_0 \}$.  For $G_0 = G_\Gamma$ we have
$\z(\g_0) = \{x \in \t \mid \gamma(x) = 0 \text{ for all } \gamma \in \Gamma\}$ and
$\z(\g_0)^\reg = \{x \in \t \mid \gamma(x) = 0 \text{ for all } \gamma \in \Gamma \text{ and } \alpha(x) \ne 0 \text{ for all } \alpha \in \Phi \setminus \Phi_\Gamma\}$.

We explain the conjugacy classes of Levi subgroups for $G$ of classical type explicitly.
The $W$-orbits of subsets $\Gamma$ of $\Pi$ are determined by the Dynkin type of $\Phi_\Gamma$ with root lengths taken into account and
with care needed to not identify $\sD_2$ with $\sA_1 + \sA_1$ or
$\sD_3$ with $\sA_3$; except in the case where $G$ has type $\sD_n$ and $\Phi_\Gamma$ has type $\sA_{i_1-1}+\dots+\sA_{i_s-1}$ with all $i_j$ even and $i_1+\dots+i_s=n$, in which case there are two $W$-orbits.

Below we give these standard Levi subgroups explicitly, up to conjugacy, when $T$ is the maximal torus of diagonal matrices in $G$ and $B$ is the
Borel subgroup of upper triangular matrices in $G$.
To do this we define an {\em inc-sequence} to be a sequence $\bi = (i_1, \dots, i_s)$
with $i_j \in \Z_{>0}$, $i_1 \le \dots \le i_s$.
For an inc-sequence $\bi$ we define $|\bi| = i_1 + \dots + i_s$.

It is easier to explain the standard Levi subgroups
of $\GL_n$ than those of $\SL_n$, and these are of the form
\begin{equation} \label{e:levigl}
	\GL_{\bi} := \GL_{i_1} \times \dots \times \GL_{i_s},
\end{equation}
for $\bi$ an inc-sequence with $|\bi| = n$.   The standard Levi subgroups of $\SL_n$ are given by
\begin{equation} \label{e:levisl}
	\SL_\bi := \GL_\bi \cap \SL_n
\end{equation}
for $\bi$ an inc-sequence with $|\bi| = n$.
For $\bj$ an inc-sequence with $|\bj| = n$, we have that $\SL_{\bi}$ is conjugate to
$\SL_{\bj}$ if and only if $\bi = \bj$.

The standard Levi subgroups of $\Sp_{2n}$
are of the form
\begin{equation} \label{e:levisp}
	\Sp_{\bi,2n_\bi} := \GL_{i_1} \times \dots \times \GL_{i_s} \times \Sp_{2n_\bi},
\end{equation}
for $\bi$ an inc-sequence with $|\bi| \le n$, and with $n_\bi := n-|\bi|$.  For $\bj$ an inc-sequence with $|\bj| \le n$,
we have that $\Sp_{\bi,2n_\bi}$ is conjugate to
$\Sp_{\bj,2n_{\bj}}$ if and only if $\bi = \bj$.

The standard Levi subgroups of $\SO_N$
are of the form
\begin{equation} \label{e:leviso}
	\SO_{\bi,N_\bi} := \GL_{i_1} \times \dots \times \GL_{i_s} \times \SO_{N_\bi},
\end{equation}
where $\bi$ is an inc-sequence with $2|\bi| \le N$, and with $N_\bi := N-2|\bi| \ne 2$.
In the case that $N=2|\bi|$
and all $i_j$ are even, there are two conjugacy classes of Levi subgroups of the form $\SO_{\bi,0}$.
In all other cases, for an inc-sequence $\bj$ with $2|\bj| \le N$ and $N-2|\bj| \ne 2$,
we have that $\SO_{\bi,N_\bi}$ is conjugate to
$\SO_{\bj,N_{\bj}}$ if and only if $\bi = \bj$.
We note that the case $N_\bi = 2$ is excluded as in this case we have $\SO_2 \cong \GL_1$, and
then  $\SO_{\bi,2}$ is conjugate to $\SO_{\bar \bi,0}$,
where $\bar \bi$ is obtained from $\bi$ by appending a 1 to the start.

We write $\gl_{\bi}$, $\sl_{\bi}$, $\sp_{\bi,2n_{\bi}}$ and $\so_{\bi,N_\bi}$ for the Lie algebras of $\GL_\bi$, $\SL_{\bi}$, $\Sp_{\bi,2n_\bi}$ and $\SO_{\bi,N_\bi}$ respectively.

\subsection{Nilpotent orbits}

We refer to $G$-orbits of nilpotent elements in $\g$ as nilpotent $G$-orbits.
We recap the parameterization of nilpotent orbits for $G$ of classical type, as can be found
for instance in \cite[Chapter~5]{CM}.
To do this we define a {\em partition} $\lambda$ to be a sequence
$\lambda = (\lambda_1, \lambda_2, \dots )$ with finitely many non-zero $\lambda_i \in \Z_{\ge 0}$,
and $\lambda_i \ge \lambda_{i+1}$ for each $i$.
Often we write $\lambda = (\lambda_1,\dots,\lambda_r)$
where $r$ is maximal such that $\lambda_r \ne 0$.
We say that $\lambda$ is a partition of $|\lambda|:= \lambda_1 + \dots + \lambda_r$.

For $G = \SL_n$ the nilpotent orbits are classified by
their Jordan type, which is the partition $\lambda$ of $n$ giving the
sizes of the Jordan blocks of an element in the orbit.

For $G = \Sp_{2n}$ the nilpotent orbits are classified by
their Jordan type, as each nilpotent $\GL_{2n}$-orbit that
intersects $\sp_{2n}$ does so in a single $\Sp_{2n}$-orbit.
The nilpotent $\GL_{2n}$-orbits that have a non-empty intersection with $\sp_{2n}$ are those with Jordan type
$\lambda$, where $\lambda$ is a partition of $2n$ such that every odd part of $\lambda$ occurs with even multiplicity; we refer to such partitions
as {\em symplectic partitions}.

For $G = \SO_{N}$, there is a minor complication as the intersection of a $\GL_{N}$-orbit
with $\so_{N}$ may split into two $\SO_{N}$-orbits; such an intersection is
a single $\mathrm O_N$-orbit, but splits into two $\SO_N$-orbits if the centralizer in $\mathrm O_N$
of an element of the orbit is contained in $\SO_N$. The Jordan types $\lambda$ of the
$\GL_{N}$-orbits having non-empty intersection with $\so_N$  are those
such that every positive even part of $\lambda$ occurs with even multiplicity; we refer to such partitions
as {\em orthogonal partitions}.
This intersection is a single $\SO_{N}$-orbit
except when $\lambda$ is a very even partition, i.e.\ when all parts of $\lambda$ are even. For a very even partition $\lambda$, it is customary to label the two $\SO_{N}$-orbits by Roman numerals I and II.

For $G$ of classical type and $\lambda$ a partition corresponding to a nilpotent $G$-orbit in $\g$, we
use the notation $\O_\lambda$ for this nilpotent orbit.  When this notation is used it is always clear
from the context which $G$ is being considered, and in the case when $G = \SO_N$ and $\lambda$ is
a very even partition it refers to a fixed choice of one of the two orbits corresponding to this partition.

\subsection{Lusztig--Spaltenstein induction and rigid nilpotent orbits} \label{ss:LSind}

We recap Lusztig--Spaltenstein induction of nilpotent orbits as introduced in \cite{LS}.
In fact in \cite{LS} the theory was developed for unipotent classes, but is also
valid for nilpotent classes, see for instance \cite[Chapter 7]{CM}. We mention that this procedure has been extended from nilpotent orbits to equivariant covers of orbits in \cite[\textsection 2.3]{LMM}.

Let $G_0$ be a Levi subgroup of $G$ and choose a parabolic subgroup $P$ with $G_0$ as
a Levi factor.  Then we have the decomposition $\p = \g_0 \oplus \u_P$ where $\u_P$ is the
Lie algebra of the unipotent radical of $P$.
Given a nilpotent $G_0$-orbit $\O_0$, there is a unique nilpotent $G$-orbit which
has dense intersection with $\O_0 + \u_P$.  We denote this orbit by $\Ind_{\g_0}^\g(\O_0)$, and say that $\O$ is
{\em Lusztig--Spaltenstein induced} from $\O_0$. As the notation indicates, Lusztig--Spaltenstein induction depends only on the Levi subalgebra $\g_0$, and not on the parabolic
subalgebra $\p$ containing $\g_0$.

We note that $G$ acts on the set of pairs $(\g_0,\O_0)$, where $\g_0$ is the Lie algebra of a Levi subgroup $G_0$ of $G$
and $\O_0\subseteq \g_0$ is a nilpotent $G_0$-orbit.  We denote the $G$-orbit of $(\g_0,\O_0)$ by $[\g_0,\O_0]$ and note that
the induced orbit $\Ind_{\g_0}^\g(\O_0)$ does not depend on the choice of representative for this orbit.
We refer to $[\g_0,\O_0]$ as an {\em induction datum} for the induced orbit $\O = \Ind_{\g_0}^\g(\O_0)$.  Since there
are finitely many conjugacy classes of Levi subgroups of $G$ and finitely many nilpotent orbits for each
Levi subgroup, there are finitely many induction data.

We say that a nilpotent $G$-orbit $\O$ is {\em rigid} in $\g$ if it cannot be obtained by induction
from a nilpotent orbit in a proper Levi subalgebra. We say that the induction datum $[\g_0, \O_0]$ is {\em rigid}
if $\O_0$ is rigid in $\g_0$.	
For a fixed nilpotent $G$-orbit $\O$, we define the {\em set of rigid induction data for} $\O$
to be
$$
\cR(\g,\O) := \{[\g_0,\O_0] \mid [\g_0, \O_0] \text{ is a rigid induction datum for } \O\}.
$$

To state what the rigid induction data are, we need to know
the rigid nilpotent orbits in Levi subalgebras of $\g$.
For this we note that a nilpotent orbit $\O_0$ in a Levi subalgebra $\g_0$ of $\g$ is rigid if and only
if the intersection of $\O_0$ with any simple factor $\h$ of $\g_0$ is rigid in $\h$.  Below
we explain what the rigid induction data are for $G$ simple of classical type.  For this
we use the notation for Levi subgroups given in \eqref{e:levisl}, \eqref{e:levisp} and \eqref{e:leviso}.

For $G =\SL_n$, the only rigid nilpotent orbit is the zero orbit.
Thus the rigid induction data are $[\sl_\bi, \{0\}]$ for $\bi= (i_1,\dots,i_s)$ an inc-sequence
with  $|\bi| = n$.  We add here that it is well-known that $\Ind_{\sl_{\bi}}^{\sl_n}(\{0\}) = \O_{\bi^{\mathrm t}}$, where
$\bi^{\mathrm t}$ is the {\em dual} of $\bi$, defined by $(\bi^t)_j = |\{k \in \{1,\dots,s\} \mid i_k \ge j\}|$.
In particular, this implies that there is a unique rigid induction datum for each nilpotent orbit $\O$ in $\g = \sl_n$.

For $G = \Sp_{2n}$, a symplectic partition $\lambda$ of $2n$ is the Jordan type of a rigid nilpotent orbit if and only if
$\lambda_i - \lambda_{i+1} \le 1$ for all $i$
and even $\lambda_i$ do not have multiplicity $2$ in $\lambda$; we refer to such $\lambda$ as a {\em rigid symplectic partition}.
The rigid induction data are $[\sp_{\bi,2n_\bi}, \{0\} \times \O_0]$ for $\bi = (i_1,\dots,i_s)$ an inc-sequence
with $|\bi| \le n$ and $\O_0$ a rigid nilpotent orbit in $\sp_{2n_\bi}$.  To clarify notation we
explain that $\{0\} \times \O_0$ denotes the nilpotent $\Sp_{\bi,2n_\bi}$-orbit in $\sp_{\bi,2n_\bi}$
that projects to $\{0\}$ in each of the $\gl_{i_j}$ factors for $j = 1,\dots,s$ and projects to $\O_0$
in the $\sp_{2n_\bi}$ factor.

For $G = \SO_{N}$, an orthogonal partition $\lambda$ of $N$ is the Jordan type of a rigid nilpotent orbit if and only if
$\lambda_i - \lambda_{i+1} \le 1$ for all $i$
and odd $\lambda_i$ do not have multiplicity $2$ in $\lambda$; we refer to such $\lambda$ as a {\em rigid orthogonal partition}.
The rigid induction data are $[\so_{\bi,N_\bi}, \{0\} \times \O_0]$ for  $\bi = (i_1,\dots,i_s)$ an inc-sequence
with $2|\bi| \le N$ and $2|\bi| \ne N-2$, and $\O_0$ a rigid nilpotent orbit in $\so_{N_\bi}$.

\subsection{Sheets and decomposition classes} \label{ss:sheets}
We recap some aspects of the theory of sheets and decomposition classes.  The material we
cover here is contained in \cite{BK} and \cite{Bo}, we also refer to \cite[\S3.1]{PT}
for an overview with references.

For each induction datum $[\g_0, \O_0]$
the  corresponding {\em decomposition class} is
$$
\cD(\g_0, \O_0) := G\cdot (\z(\g_0)^{\reg} + \O_0).
$$
Decomposition classes give a stratification
$$
\g = \bigsqcup_{[\g_0,\O_0]} \cD(\g_0, \O_0)
$$
where the (disjoint) union is taken over the finitely many induction data for all nilpotent orbits in $\g$.

The {\em rank strata} of $\g$ are defined to be
$$
\g^{(j)} := \{x \in \g \mid \dim G \cdot x = j\}
$$
for $j \in \Z_{\ge 0}$.
The {\em sheets} of $\g$ are the irreducible components of the rank strata.

Each decomposition class $\cD$ is irreducible and there exists $j_\cD \in \Z_{\ge 0}$
such that $\cD \sub \g^{(j_\cD)}$. Therefore, each sheet
contains a  dense decomposition class.
Given a decomposition class $\cD$ we let $\bar{\cD}^{\reg} := \bar{\cD} \cap \g^{(j_\cD)}$
be the union of the orbits of maximal dimension in $\bar \cD$.
Given decomposition classes $\cD_1 := \cD(\g_1, \O_1),$ and $\cD_2 := \cD(\g_2, \O_2)$ we have
$\cD_1 \sub \bar {\cD}_2^{\reg}$ if and only if there are representatives $(\g_1, \O_1)$ and $(\g_2,\O_2)$ such that
$\g_2 \sub \g_1$ and $\O_1 = \Ind_{\g_2}^{\g_1}(\O_2)$.  From this we can deduce that
the sheets of $\g$ are precisely the $\bar{\cD}^{\reg}(\g_0, \O_0)$:=$\bar{\cD(\g_0,\O_0)}^{\reg}$ as we vary
over all rigid induction data $[\g_0, \O_0]$, and moreover that $\Ind_{\g_0}^\g \O_0$ is the unique nilpotent
orbit lying in $\bar{\cD}^{\reg}(\g_0, \O_0)$. Finally, we note for future reference that $\dim \bar{\cD}^{\reg}(\g_0, \O_0) = \dim\z(\g_0)+\dim \Ind_{\g_0}^\g \O_0$ (see \cite[Satz 4.5]{BK}).

\subsection{Primitive ideals of universal enveloping algebras and their central characters} \label{ss:primitive}

We recall the required background on primitive ideals and central characters of the universal
enveloping algebra $U(\g)$ of $\g$.  We refer to \cite{BJ} as a general reference
on primitive ideals and to \cite[Chapter 1]{HuBGG} as a general reference for
central characters.

An ideal $I$ of $U(\g)$ is called {\em primitive} if $I = \Ann_{U(\g)}(E)$ for some simple $U(\g)$-module $E$, 
where $\Ann_{U(\g)}(E)$ denotes the annihilator of $E$ in $U(\g)$.
We denote the spectrum of primitive ideals by $\Prim U(\g)$.  An ideal $I$ of $U(\g)$
is called {\em completely prime} if $U(\g)/I$ is a domain.  The spectrum of completely
prime primitive ideals is denoted $\Prim^1 U(\g)$.

The Poincar\'e--Birkhoff--Witt (PBW) filtration of $U(\g)$ induces a filtration on an ideal $I$
of $U(\g)$ and we write $\gr I \sub \gr U(\g) \cong S(\g)$ for the associated graded ideal.
We identify $S(\g)$ with $\C[\g]$ via our choice of nondegenerate invariant symmetric bilinear form
$(\cdot\,,\cdot)$ on $\g$.  We recall that the {\em associated
	variety} $\cVA(I)$ of an ideal $I$ of $U(\g)$ is the vanishing locus in $\g$ of $\gr I \sub \C[\g]$.
As mentioned in the introduction it is known that for $I \in \Prim U(\g)$ we have that $\cVA(I)$ is
the closure of a nilpotent orbit in $\g$ by Joseph's irreducibility theorem \cite{JoAV}.
For a nilpotent $G$-orbit $\O$ we define $\Prim_\O U(\g) = \{I \in \Prim(U(\g)) \mid \cVA(I) = \bar \O\}$
and $\Prim^1_\O U(\g) = \Prim^1 U(\g) \cap \Prim_\O U(\g)$.

Let $P$ be a parabolic subgroup of $G$, let $G_0$ be a Levi factor of
$P$, and let $U_P$ be the unipotent radical of $P$.  We have the decomposition $\p = \g_0 \oplus \u_P$.
Let $I_0 \in \Prim U(\g_0)$ and choose a simple $U(\g_0)$-module $E$ such that
$I_0 = \Ann_{U(\g_0)}(E) \sub U(\g_0)$.

We define the induced ideal
$$
\Ind_{\p}^\g(I_0) := \Ann_{U(\g)}(U(\g) \otimes_{U(\p)} E)
$$
where $E$ is viewed as a $U(\p)$-module via the projection $U(\p)\onto U(\g_0)$.
It is well-known and straightforward to show that the induced ideal $\Ind_{\p}^\g(I_0)$
can also be described as the largest
two-sided ideal contained in the left ideal $U(\g)(I_0+\u_P)$.
From this alternative description it follows that $\Ind_{\p}^\g(I_0)$ does not depend
on the choice of simple $U(\g_0)$-module $E$ with $I_0 = \Ann_{U(\g_0)}(E)$.
(For a discussion of the relationship between $\Ind_{\p}^\g$ and the choice of parabolic subalgebra $\p$ with Levi factor $\g_0$, see \cite[\S 5.4]{BJ}.)

The induced ideal $\Ind_{\p}^\g(I_0)$ is in general not primitive.  However,
if $I_0 \in \Prim^1 U(\g_0)$, then $\Ind_{\p}^\g(I_0) \in \Prim^1 U(\g)$ as is proved in
\cite[Theorem~3.1]{Co}.

Although $\Ind_\p^\g(I_0)$ is not primitive in general
its associated variety is the closure of a nilpotent orbit and is given by
\begin{equation} \label{e:indva}
	\cVA(\Ind_\p^\g(I_0)) = \bar{\Ind_{\g_0}^\g(\O_0)}
\end{equation}
where $\bar{\O}_0 = \cVA(I_0)$.  A proof of \eqref{e:indva} using results of Losev is explained within \cite[\S5]{BG}; when $I_0\in\Prim^1 U(\g_0)$ see also \cite[Lemma 7.3]{BJ}.

Let $\O_0$ be a nilpotent $G_0$-orbit and let $\O = \Ind_{\g_0}^\g(\O_0)$.
From the discussion above it follows that parabolic induction restricts to a map 	
$$
\Ind_\p^\g : \Prim^1_{\O_0} U(\g_0) \to \Prim^1_{\O} U(\g).
$$

We write $Z(\g)$ for the centre of $U(\g)$, and refer to a character $\chi : Z(\g) \to \C$
as a {\em central character} of $U(\g)$.  By Quillen's lemma we know that
$Z(\g)$ acts on a simple $U(\g)$-module $E$ by a central character
$\chi_E : Z(\g) \to \C$.  Thus, $\Ann_{U(\g)}(E) \cap Z(\g)$ is an ideal
of $Z(\g)$ of codimension $1$. Setting $I=\Ann_{U(\g)}(E)$, we write $\chi_I : Z(\g) \to \C$ for the corresponding central character.
It is then clear that $\chi_I = \chi_{E'}$ for any simple $U(\g)$-module $E'$ with $I = \Ann_{U(\g)}(E')$.

We write $U(\g)_0$ for the subalgebra of $U(\g)$ given by
the zero weight space for the adjoint action of $\t$ on $U(\g)$.  Then
$Z(\g) \sub U(\g)_0$.  We have that $U(\g)_0 = S(\t) \oplus (U(\g)\u \cap U(\g)_0)$ with $U(\g)\u \cap U(\g)_0$ being
a 2-sided ideal of $U(\g)_0$, where we recall that $\u$ is the Lie algebra of the unipotent radical of $B$.
Thus the projection $U(\g)_0 \to S(\t)$ defined from the decomposition $U(\g)_0 = S(\t) \oplus (U(\g)\u \cap U(\g)_0)$ is a
homomorphism.  The homomorphism $Z(\g) \to S(\t)$ given
by restriction is known to be an isomorphism onto its image and this image is
$S(\t)^{W,\bullet} := \{u \in S(\t) \mid w \bullet u = u \text{ for all } w \in W\}$, where to define the action of $W$ on $S(\t)$ we identity $S(\t)$ with the algebra of polynomial functions on $\t^{*}$ and
recall the dot action on $\t^*$ as defined in \S\ref{ss:notn}.
This isomorphism $\psi : Z(\g) \to S(\t)^{W,\bullet}$
is called the Harish-Chandra isomorphism.  Thus as $S(\t)^{W,\bullet} \cong \C[\t^*/(W,\bullet)]$, where $\t^{*}/(W,\bullet)$ denotes
the quotient of $\t^{*}$ under the dot-action of $W$, we may identify central characters of primitive ideals
of $U(\g)$ with $(W,\bullet)$-orbits in $\t^*$.
We write
$$
\ch : \Prim U(\g) \to \t^*/(W,\bullet)
$$
for this central character map.

In fact the map $\ch$
can be extended to all ideals $I$ of $U(\g)$ for which
$I \cap Z(\g)$ is of codimension $1$ in $Z(\g)$; such ideals are said
to {\em admit a central character}.
Given a highest weight $U(\g)$-module $M$,
we can see that the action of $Z(\g)$ on a highest weight vector of $M$
must be given by a central character, and thus $Z(\g)$ acts on all of $M$ by this central character.
A useful observation is that the map $\ch$ is set up so that if $I = \Ann_{U(\g)}(M)$ where $M$ is a highest
weight $U(\g)$-module with highest weight $\lambda \in \t^*$ with respect to $\b$,
then $\ch(I) = W \bullet \lambda$.

We now suppose that $P$ is a standard parabolic subgroup of $G$ and $G_0$ is the standard Levi factor
of $P$.
We note that $B_0 := B \cap G_0$ is a Borel subgroup of $G_0$.
The Weyl group of $G_0$ with respect to $T$ is denoted $W_0$.

We have the Harish-Chandra isomorphism $\psi_0 : Z(\g_0) \to S(\t)^{W_0,\bullet}$; here we note that
we can use the same $\rho$ for $G_0$ as was used for $G$ to define the dot action of $W_0$.
Therefore, we have a central character map
$$
\ch_0 : \Prim U(\g_0) \to \t^*/(W_0,\bullet).
$$
We write
$$
\pi : \t^*/(W_0,\bullet) \onto \t^*/(W,\bullet)
$$
for the projection map.

Let $I_0 \in \Prim U(\g_0)$.  By Duflo's theorem we can assume $I_0 = \Ann_{U(\g_0)}(L_0(\lambda))$ for some $\lambda \in \t^*$,
where $L_0(\lambda)$ denotes the simple highest weight $U(\g_0)$-module with highest weight $\lambda$ with respect
to the Borel subalgebra $\b_0$.  Then we have $\ch_0(I_0) = W_0 \bullet \lambda$.  Let $I = \Ind_\p^\g(I_0)$, so we have
$I = \Ann_{U(\g)}(U(\g) \otimes_{U(\p)} L_0(\lambda))$.  Since $U(\g) \otimes_{U(\p)} L_0(\lambda)$ is a highest weight $U(\g)$-module
of highest weight $\lambda$, we see that $\ch(I) = W \bullet \lambda$.  Summarising this discussion we have that
\begin{equation} \label{e:centralchars}
	\pi(\ch_0(I_0)) = \ch(\Ind_\p^\g(I_0)).
\end{equation}

\section{On Lusztig--Spaltenstein induction for simple algebraic groups of classical type}
\label{S:LSinduction}

For this section we restrict to the case where $G$ is simple of classical type. Our main goal is to prove the following theorem, which is a key ingredient in our proof of
Theorem~\ref{T:surj}.
\begin{theorem}\label{T:nonconj}
	Let $G$ be a simple algebraic group of classical type and let $\O$ be a nilpotent $G$-orbit.  Let $[\g_1,\O_1]$
	and $[\g_2,\O_2]$ be rigid induction data for $\O$, and let $(\g_1,\O_1)$ and $(\g_2,\O_2)$ be representatives thereof.
	Suppose that $\g_1 \sub \g_2$.  Then $\g_1=\g_2$ and $\O_1 = \O_2$.
\end{theorem}

The proof of this result is straightforward for $G=\SL_N$, due to the uniqueness of rigid induction data in that case, and so from now on we assume $G=\Sp_{2n}$ or $G=\SO_N$. The main step in our proof will be a combinatorial argument which we can apply when $\O$ corresponds to a partition of the form $(2r, 2r-2,\ldots,4,2)$ for $G=\Sp_{r(r+1)}$ or a partition of the form $(2r-1,2r-3,\ldots,3,1)$ for $G=\SO_{r^2}$. An argument involving \cite[Lemma 8.4]{To} then allows us to prove the result for all partitions.

Central to our approach to the combinatorial argument is the Kempken--Spaltenstein algorithm (or KS-algorithm for short)
which was introduced in \cite[\S3.1]{PT} and determines
the rigid induction data attached to a nilpotent orbit in $\g$.
We give an overview of this algorithm for the case $G = \Sp_{2n}$ and then
explain the minor adaptations required for $G = \SO_N$.  We refer to
\cite[\S 3--4]{PT} for all the claims about the KS-algorithm made below.

We begin with a symplectic partition $\lambda=(\lambda_1,\ldots,\lambda_r)$ of $2n$.  Then, as we
define and explain next, the KS algorithm makes a sequence of reductions to $\lambda$ (called type 1 reductions and type 2 reductions) corresponding to an admissible sequence
$\bi = (i_1,\dots,i_s)$, where $1 \le i_j \le r$ for each $j = 1,\dots,s$.

We say that a {\em type 1 reduction} can be made to $\lambda$ at position $i$ if
$\lambda_i \ge \lambda_{i+1} + 2$. Applying this reduction gives a partition $\lambda^i$ of $2n-2i$ defined by
$$
\lambda^i_j = \begin{cases} \lambda_j-2 & \text{if $j \le i$,} \\
	\lambda_j & \text{if $j > i$}. \end{cases}
$$

We say that a {\em type 2 reduction} can be made to $\lambda$ at position $i$ if $\lambda_i$ is even and
$\lambda_{i-1} > \lambda_{i} = \lambda_{i+1} > \lambda_{i+2}$ (here we consider the inequality $\lambda_0>\lambda_1$ to automatically hold). Applying this reduction gives a partition $\lambda^i$ of $2n-2i$ defined by
$$
\lambda^i_j = \begin{cases} \lambda_j-2 & \text{if $j < i$,} \\
	\lambda_j-1 & \text{if $i \le j \le i+1$,} \\
	\lambda_j & \text{if $j > i+1$}. \end{cases}
$$

We note that at a given position $i$, it is not possible that both a type 1 reduction and a type 2 reduction can be made to $\lambda$,
so there is no ambiguity in the notation $\lambda^i$.  We say that $i$ is an {\em admissible position}
for $\lambda$ if a type 1 or type 2 reduction can be made to $\lambda$ at position $i$.

We define recursively what it means for $\bi = (i_1,\dots,i_s)$ to be an {\em admissible sequence} for $\lambda$. First we introduce the notation that for $k = 1,\dots,s$ we write $\bi_k = (i_1,\dots i_k)$.  We say that $\bi_1$ is an admissible sequence for $\lambda$ if $i_1$ is an admissible position for $\lambda$; we then set $\lambda^{\bi_1} = \lambda^{i_1}$.
We then recursively say that $\bi_k$ is an admissible sequence for $\lambda$, and define $\lambda^{\bi_k} := (\lambda^{\bi_{k-1}})^{i_k}$, if $i_k$ is an admissible position for $\lambda^{\bi_{k-1}}$. We say that $\bi$ is a {\em maximal admissible sequence} if there is
no admissible position for $\lambda^\bi$.

To summarise, the KS-algorithm takes as input a symplectic partition $\lambda$
and an admissible sequence $\bi$, and outputs the symplectic partition $\lambda^\bi$.
The sequence of type 1 and type 2 reductions given by $\bi$ are used to determine $\lambda^\bi$.

To explain the significance of the outcome of this algorithm we require some notation.
We fix an admissible sequence $\bi = (i_1,\dots,i_s)$ for $\lambda$. We write $\O_\lambda$
for the nilpotent orbit in $\g = \sp_{2n}$ with Jordan type $\lambda$ and $\O_{\lambda^\bi}$
for the nilpotent orbit in $\sp_{2n_\bi}$ with Jordan type $\lambda^\bi$.
The upshot of the KS-algorithm is that 
$[\sp_{\bi,2n_\bi}, \{0\} \times \O_{\lambda^\bi}]$ is an induction datum for $\O_\lambda$, see \cite[Proposition~7]{PT}.
Here we note that although in \eqref{e:levisp} we just consider
the definition of $\sp_{\bi,2n_\bi}$ for $\bi$ an inc-sequence, the definition makes sense for any
admissible sequence $\bi$.
Furthermore, $\{0\} \times \O_{\lambda^\bi}$ is a rigid nilpotent orbit in $\sp_{\bi,2n_\bi}$
if and only if $\bi$ is a maximal admissible sequence.  Moreover, every rigid induction datum for $\O_\lambda$
is of the form $[\sp_{\bi,2n_\bi}, \{0\} \times \O_{\lambda^\bi}]$ for some
maximal admissible sequence $\bi$.  Thus to obtain all rigid induction data for $\O_\lambda$
we can apply the KS-algorithm for all maximal admissible sequences for $\lambda$.

Now we discuss the modification needed to the algorithm for $G = \SO_N$.  This is simply that
we replace ``even'' with ``odd'' in the description of a type 2 reduction.  With this version of the algorithm
all of the results given in the symplectic case above have analogues for the orthogonal case.  In particular,
a maximal admissible sequence $\bi = (i_1,\dots,i_s)$ gives a rigid induction datum
$[\so_{\bi,N_\bi}, \{0\} \times \O_{\lambda^\bi}]$ for $\O_\lambda$, where
now $\O_{\lambda}$ is a nilpotent orbit in $\so_N$ and $\O_{\lambda^\bi}$ is a nilpotent orbit in
$\so_{N_\bi}$; and, moreover, all rigid induction data are obtained in this way.

There is a subtlety in the orthogonal case
that if there are two $\SO_N$-orbits of Jordan type $\lambda$, then either there are two $\SO_{N_\bi}$-orbits
of Jordan type $\lambda^\bi$, or there are two conjugacy classes of Levi subgroups $\SO_{\bi,N_\bi}$.  In these
cases it is implicit that we have to choose the correct $\SO_{\bi,N_\bi}$-orbits with Jordan type $\lambda^\bi$, or the
correct $\SO_N$-conjugacy class of Levi subalgebras of the form $\so_{\bi,N_\bi}$.

We recall a property of the KS-algorithm that is important in this paper
as it reduces the number of maximal admissible sequences that have to be considered.  In the
statement $\lambda$ is a symplectic partition or an orthogonal partition depending on which case is being
considered.

\begin{ksprop} \label{KS:order}
	Any reordering of the entries in a maximal admissible
	sequence $\bi$ for $\lambda$ gives a maximal admissible sequence for $\lambda$, and
	the partition $\lambda^\bi$ does not depend on the order the entries of $\bi$.
\end{ksprop}

This property follows from \cite[Corollary~7]{PT} and \cite[Proposition~7]{PT}.
We remark here that for the orthogonal case the statement of \cite[Corollary~7]{PT} excludes the
possibility that $N_\bi = 2$.  We note, however, that in the case $N_\bi = 2$, we have $\lambda^\bi = (1^2)$ and
we can make a final type 2 reduction at position $1$, and thus $\bi$ is not a maximal admissible
sequence.

As a consequence of KS-property~\ref{KS:order} we can
restrict to considering maximal admissible sequences $\bi$ that are inc-sequences, and
we refer to such $\bi$ as a {\em maximal admissible inc-sequence}.

We now have most of the machinery needed to prove the following lemma.
\begin{lemma}  \label{L:specialpartitions}
	Let $r \in \Z_{>0}$.
	\begin{enumerate}
		\item[{\em (a)}] Let $G = \Sp_{r(r+1)}$ and let $\lambda = (2r,2r-2,\dots,4,2)$.  Let $[\g_1,\O_1]$ and 
		$[\g_2,\O_2]$ be rigid induction data for $\O_\lambda$, and let $(\g_1,\O_1)$ and $(\g_2,\O_2)$ be representatives thereof.
		Suppose that $\g_1 \sub \g_2$.  Then $\g_1=\g_2$ and $\O_1 = \O_2$.
		\item[{\em (b)}] Let $G = \SO_{r^2}$ and let $\lambda = (2r-1,2r-3,\dots,3, 1)$.  Let $[\g_1,\O_1]$ and
		$[\g_2,\O_2]$ be rigid induction data for $\O_\lambda$, and let $(\g_1,\O_1)$ and $(\g_2,\O_2)$ be representatives thereof.
		Suppose that $\g_1 \sub \g_2$.  Then $\g_1=\g_2$ and $\O_1 = \O_2$.
	\end{enumerate}
\end{lemma}

Before embarking on the proof of Lemma~\ref{L:specialpartitions} we choose to first include an example.  This should be helpful for the reader
as the proof of Lemma~\ref{L:specialpartitions}
is combinatorial and a little technical,
and the example demonstrates some of the ideas in the proof.

\begin{exa} \label{E:642}
	We consider the nilpotent orbit with Jordan type $\lambda = (6,4,2)$ in $\sp_{12}$.  We determine
	all maximal admissible inc-sequences, showing that there are just three of them.
	Let $\bi = (i_1,\dots,i_s)$ be a maximal admissible inc-sequence.  We proceed by considering cases
	based on the values of $i_1$ and $i_2$.
	
	\smallskip
	\noindent
	{\em Case $i_1 = 3$.}  We have $\lambda^{\bi_1} = \lambda^3 = (4,2,0)$ and $i_2 \ge 3$.  Since $\lambda^3$ is not
	a rigid symplectic partition and there are no admissible positions $i \ge 3$ for $\lambda^3$, we
	deduce that $i_1 = 3$ is not possible.
	
	\smallskip
	\noindent
	{\em Case $i_1 = 2$.}  We have $\lambda^{\bi_1} = \lambda^2 = (4,2,2)$ and $i_2 \ge 2$. 
	
	{\em Case $(i_1,i_2)=(2,3)$.}  We have $\lambda^{\bi_2} = \lambda^{(2,3)} = (2,0,0)$ and $i_3 \ge 3$.
	Since $\lambda^{(2,3)}$ is not a rigid symplectic partition and there are no admissible positions $i \ge 3$
	for $\lambda^{(2,3)}$, we deduce that $i_2 = 3$ is not possible.
	
	{\em Case $(i_1,i_2)=(2,2)$.}  We have $\lambda^{\bi_2} = \lambda^{(2,2)} = (2,1,1)$ is a rigid symplectic
	partition.  Thus $\bi = (2,2)$ is a maximal admissible inc-sequence for $\lambda$.
	
	\smallskip
	\noindent
	{\em Case $i_1 = 1$.}  We have $\lambda^{\bi_1} = \lambda^1 = (4,4,2)$. 
	
	{\em Case $(i_1,i_2)=(1,3)$.}  We have $\lambda^{\bi_2} = \lambda^{(1,3)} = (2,2,0)$ and $i_3 \ge 3$.
	Since $\lambda^{(1,3)}$ is not a rigid symplectic partition and there are no admissible positions $i \ge 3$
	for $\lambda^{(1,3)}$, we can deduce that $i_2 = 3$ is not possible. 
	
	{\em Case $(i_1,i_2)=(1,2)$.}  We have $\lambda^{\bi_2} = \lambda^{(1,2)} = (2,2,2)$.
	The only admissible position for $(2,2,2)$ is 3, so we must have $i_3 = 3$ and then
	$\lambda^{\bi_3} = \lambda^{(1,2,3)} = (0,0,0)$ is a rigid symplectic partition.
	Thus $\bi = (1,2,3)$ is a maximal admissible inc-sequence for $\lambda$. 
	
	{\em Case $(i_1,i_2)=(1,1)$.}  We have $\lambda^{\bi_2} = \lambda^{(1,1)} = (3,3,2)$.
	The only admissible position for $(3,3,2)$ is 3, so we must have $i_3 = 3$ and then
	$\lambda^{\bi_3} = \lambda^{(1,1,3)} = (1,1,0)$ is a rigid symplectic partition.
	Thus $\bi = (1,1,3)$ is a maximal admissible inc-sequence for $\lambda$.
	
	\smallskip
	
	We have seen that there are three maximal admissible inc-sequences which we label as $\bi^1 = (1,1,3)$, $\bi^2 = (1,2,3)$ and $\bi^3 = (2,2)$.
	We label the rigid induction data corresponding to $\bi^j$ by $[\g_j,\O_j]$ for $j=1,2,3$, and these are given as follows.
	\begin{itemize}
		\item For $\bi^1$ we have $\g_1 = \gl_1 \times \gl_1 \times \gl_3 \times \sp_2$ and $\O_1$ is the zero orbit.
		\item For $\bi^2$ we have $\g_2 = \gl_1 \times \gl_2 \times \gl_3$ and $\O_2$ is the zero orbit.
		\item For $\bi^3$ we have $\g_3 = \gl_2 \times \gl_2 \times \sp_4$ and $\O_3 = \{0\}^2 \times \O_{(2,1,1)}$.
	\end{itemize}
	
	In preparation for the proof of Lemma~\ref{L:specialpartitions} it is helpful for us to observe
	that for $i,j \in \{1,2,3\}$ with $i \ne j$ there is no inclusion $\tilde \g_i \sub  \g_j$ for $\tilde \g_i$ a $G$-conjugate of $\g_i$.
	To do this we recall from \S\ref{ss:paralevi} that we may assume $\tilde \g_i$ is a standard Levi subalgebra and thus of the form described in \S\ref{ss:paralevi}.
	
	To demonstrate this lack of inclusions, first consider the standard Levi subalgebras of $\g_1$. These are of the form $\gl_1 \times \gl_1 \times \gl_{\bj^1}
	\times \sp_{\bj^2,2n_{\bj^2}}$, where $\bj^1$ and $\bj^2$ are inc-sequences with $|\bj^1| = 3$ and $|\bj^2|\le 1$, and $n_{\bj^2}=1-|\bj^2|$.
	Therefore, $\gl_{\bj^1}$ is one of $\gl_3$, $\gl_1 \times \gl_2$ or $\gl_1 \times \gl_1 \times \gl_1$, and $\sp_{\bj^2,2n_{\bj^2}}$
	is one of $\sp_2$ or $\gl_1$.
	Since none of these Levi subalgebras involves both a $\gl_2$ and a $\gl_3$, we see as above that no $G$-conjugate of $\g_2$ can
	be a subalgebra of $ \g_1$, and
	since none of these involve an $\sp_4$ we see that no $G$-conjugate of $\g_3$ can be a subalgebra of $ \g_1$.
	
	We can similarly consider the forms of all standard Levi subalgebras of $\g_2$ and note that none of these involve a symplectic
	subalgebra (containing a root subspace for a long root), so that no conjugates of either $\g_1$ nor $\g_3$ are contained in $\g_2$.  For $\g_3$ we can see that no standard Levi subalgebra can involve a $\gl_3$, so that no conjugate of either $\g_1$ or $\g_2$
	is contained in $\g_3$.
\end{exa}

We pick out a key property of the KS-algorithm, which is demonstrated in Example~\ref{E:642}.

\begin{ksprop}  \label{KS:2diff}
	Let $\lambda$ be a symplectic partition or an orthogonal partition, depending on which case we are considering.
	Let $\bi$ be a maximal admissible inc-sequence
	for $\lambda$, and let $i \in \Z_{\ge 1}$ with $i_1 \ge i$.
	Suppose that $\lambda_i = \lambda_{i+1} + 2$.   Then $i_1 = i$ or $i_1 = i+1$, and moreover, if $i_1 = i+1$ and $k$ is maximal such that $i_k = i+1$, then the $k$th reduction determined by $\bi$ is of type 2.
\end{ksprop}

\begin{proof}
	Since $\bi$ is a maximal admissible sequence, we must have $\lambda^\bi_i \le \lambda^\bi_{i+1}+1$
	and thus that $\lambda^{\bi_j}_i - \lambda^{\bi_j}_{i+1} < \lambda^{\bi_{j-1}}_i - \lambda^{\bi_{j-1}}_{i+1}$
	for some $j$.  The only way we can get that $\lambda^{\bi_j}_i - \lambda^{\bi_j}_{i+1}
	< \lambda^{\bi_{j-1}}_i - \lambda^{\bi_{j-1}}_{i+1}$ is for $i_j =i$, or for $i_j = i+1$ corresponding to a type 2 reduction.
	Since $\bi$ is an admissible inc-sequence we deduce that $i_1 = i$ or $i_1 = i+1$.  Furthermore, in the latter case
	there must be a type 2 reduction at position $i+1$, which implies that for the maximal $k$ such that
	$i_k = i+1$, we have that the reduction corresponding to $i_k$ is of type 2.
\end{proof}

We move on to prove Lemma~\ref{L:specialpartitions}, which covers a key special case of Theorem~\ref{T:nonconj}.

\begin{proof}[Proof of Lemma~\ref{L:specialpartitions}]
	(a) We assume that $r > 3$ as the case $r=3$ is covered in Example~\ref{E:642}, and the cases $r=1$ and $r = 2$
	can be dealt with easily.
	
	Arguing as in \S\ref{ss:paralevi} and Example~\ref{E:642}, we may assume that $\g_1$ and $\g_2$ are
	standard Levi subalgebras with
	$\g_1$ contained in $\g_2$.
	There are maximal admissible inc-sequences $\bi = (i_1,\dots,i_s)$ and $\bj = (j_1,\dots,j_t)$ such that
	$\g_1 = \sp_{\bi,2n_\bi}$ and $\O_1 = \O_{\lambda^\bi}$, and $\g_2 = \sp_{\bj,2n_\bj}$ and $\O_2 = \O_{\lambda^\bj}$. As explained in \S\ref{ss:paralevi},
	$\g_1$ is a standard Levi subalgebra of $\g_2$.  Now by considering the form of Levi subalgebras given in \eqref{e:levigl} and \eqref{e:levisp}
	we deduce that there is a function $h : \{1,\dots,s\} \to \{1,\dots,t,t+1\}$ such that for each $k \in \{1,\dots,t\}$ we have
	$\sum_{l \in h^{-1}(k)} i_l = j_k$.
	The idea here is that the inclusion
	$\g_1 \sub \g_2$ is given by inclusions $\bigoplus_{l \in h^{-1}(k)} \gl_{i_l} \sub \gl_{j_k}$ and an inclusion $( \bigoplus_{l \in h^{-1}(t+1)} \gl_{i_l} ) \oplus \sp_{2n_\bi}
	\sub \sp_{2n_\bj}$.
	Moreover, for such a function $h$ we may (and shall) assume that for $l \le l'$ with $i_l = i_{l'}$ we have $h(l) \le h(l')$.
	
	We prove by induction on $k$ that $\bi_k = \bj_k$ and that $h(l) = l$ for $l = 1,\dots,k$.  For the case
	$k = s$ this implies that $\bi = \bj$ and thus that $\g_1 = \g_2$ and $\O_1 = \O_2$ as required
	
	We could take our base case to be $k=0$, which is trivial, but we find it more instructive to
	include $k =1$ for the base case.  We are led to consider the possibilities for $(i_1,i_2,i_3)$ and the possibilities for $(j_1,j_2,j_3)$.
	The analysis of cases below is very similar to that in Example~\ref{E:642}.
	
	By KS-property~\ref{KS:2diff}, we have $i_1=1$ or $i_1 = 2$.  We consider these cases separately.
	
	\noindent
	{\em Case $i_1 = 2$.}   We must have $i_2 = 2$ by KS-property~\ref{KS:2diff}.  Then we have $\lambda^{(2,2)} = (2r-4,2r-5,2r-5,2r-6,2r-8,\dots,2)$ and we see
	that $i_3 \ge 4$.
	
	\noindent
	{\em Case $i_1 = 1$.}  We have $\lambda^1 = (2r-2,2r-2,2r-4,2r-6,\dots,2)$.  By using KS-property~\ref{KS:2diff}
	we see that the only possibilities for $i_2$ are $i_2 = 1$, $i_2 = 2$ or $i_2=3$,
	and we consider these cases.
	
	{\em Case $(i_1,i_2) = (1,3)$.}  Using KS-property~\ref{KS:2diff} again we have that $i_3 = 3$. Then we have 
	$\lambda^{(1,3,3)} = (2r-6,2r-6,2r-7,2r-7,2r-8,\dots,2)$.  By considering the subsequent reductions we see that 
	$\lambda^{\bi_k}_1$ is even and  $\lambda^{\bi_k}_1 = \lambda^{\bi_k}_2 > \lambda^{\bi_k}_3$ for each $k$.
	This is not possible as $\lambda^{\bi}$ is a rigid symplectic partition.
	
	{\em Case $(i_1,i_2) = (1,2)$.}
	We have $\lambda^{(1,2)} = (2r-4,2r-4,2r-4,2r-6,\dots,2)$ and then by KS-property~\ref{KS:2diff}
	we have that $i_3 \ge 3$.
	
	{\em Case $(i_1,i_2) = (1,1)$.}
	We have $\lambda^{(1,1)} = (2r-3,2r-3,2r-4,2r-6,\dots,2)$ and then by KS-property~\ref{KS:2diff}
	we have that $i_3 \ge 3$. 
	
	Summarising this case analysis we deduce that $(i_1,i_2,i_3)$ is one of:  
	$(2,2,a)$ where $a \ge 4$; 
	$(1,2,b)$ where $b \ge 3$; or 
	$(1,1,c)$ where $c \ge 3$.
	Similarly $(j_1,j_2,j_3)$ must be one of these possibilities.
	
	Suppose that $i_1 = 2$ and $j_1 = 1$. Then we have $h(l) > 1$ for all $l$, as $2 = i_1 > j_1 = 1$,
	and $\sum_{l \in h^{-1}(1)} i_l = 0 \ne j_1$.
	Thus $i_1 = 2$ and $j_1 = 1$ is not possible.
	
	Suppose that $i_1 = 1$ and $j_1 = 2$, so $(j_1,j_2,j_3)$ is of the form $(2,2,a)$ and $(i_1,i_2,i_3)$ is one of the two
	possibilities with $i_1 = 1$.  We have that $\sum_{l \in h^{-1}(1)} i_l = 2$, and thus the multiset $\{i_l \mid l \in h^{-1}(1)\}$ is either $\{1,1\}$
	or $\{2\}$.  Similarly, the multiset $\{i_l \mid l \in h^{-1}(2)\}$ is either $\{1,1\}$ or $\{2\}$.  For this to be possible, one of the following possibilities must occur:
	\begin{itemize}
		\item
		$1$ occurs with multiplicity at least 4 in $\bi$; 
		\item
		$1$ occurs with multiplicity at least 2 and 2 occurs with multiplicity at least 1 in $\bi$; or 
		\item
		$2$ occurs with multiplicity at least 2 in $\bi$. 
	\end{itemize}
	By considering the two possibilities for $(i_1,i_2,i_3)$ with $i_1 = 1$, we see that none of the above three conditions hold for $\bi$.
	Thus $i_1 = 1$ and $j_1 = 2$ is not possible.
	
	Therefore, $i_1 = j_1$. Then we must also have that $h(1) = 1$, as $\bi$ is an increasing
	sequence and, by assumption, $h(1) \le h(l)$ for any $l$ such that $i_1 = i_l$.
	This completes the base case.

	Before moving on to the inductive step, we make a couple of useful observations that we exploit several times below. Note that these observations relate to the KS-algorithm as applied to the partitions considered in the statement of the lemma, but are not properties of the KS-algorithm in general.
	\begin{obs}\label{o1}
		For any $i$ the multiplicity of $i$ in $\bi$ is at most 2,
		and if the multiplicity is 2 then the first reduction
		at position $i$ is a type 1 reduction and the second is a type 2 reduction.
	\end{obs}
	This is
	necessary as $\lambda^{\bi}_{i} \ge \lambda^\bi_{i+1}$, and for parity reasons we can have at most one
	type 2 reduction at position $i$. 
	\begin{obs}\label{o2}
		If there is a type 2 reduction at
		position $i$, then there are no reductions at position $i+1$.
	\end{obs}
	To explain this suppose that
	$i_{k'} = i$ and this corresponds to a type $2$ reduction.  Then we have
	$\lambda^{\bi_{k'}}_{i+1} = \lambda^{\bi_{k'}}_{i+2}+1$, so that $i+1$ is not an
	admissible position for $\lambda^{\bi_{k'}}$, and thus $i_{k'+1} > i+1$.
	
	Moving on to the inductive step, we assume, for some $k \le s$, that $\bi_{k-1} = \bj_{k-1}$ and $h(l) = l$ for all $l < k$.
	
	Let $i = i_k$ and $j = j_k$.  We note that $j \ge i$, because if $j < i$, then there are no $l$ such that $h(l) = k$, and thus
	$\sum_{l \in h^{-1}(k)} i_l = 0 \ne j$.  We aim to show that  $i = j$, which implies that $h(k) = k$, to complete 
	the inductive step.  Thus we assume that $j > i$ and we aim for a contradiction.  We also note that if $k = s$, then we
	must have $i = j$, so we assume that $k < s$.
	
	We proceed to consider two cases.
	
	\smallskip

	\noindent
	{\em Case $i_{k-1}<i$.}
	We first note that by Observation~\ref{o2} we cannot have $i_{k-1}=i-1$ with this reduction being of type $2$ otherwise
	$i$ is not an admissible position for $\lambda^{\bi_{k-1}}$.
	Thus in the first $k-1$ reductions given by $\bi$ there have been no reductions at position $i$ and 
	no type 2 reduction at position $i-1$, so none of these reductions have  decreased $\lambda_{i'}$ for $i' \ge i$.  Therefore,
	$(\lambda^{\bi_{k-1}}_i,\lambda^{\bi_{k-1}}_{i+1},\lambda^{\bi_{k-1}}_{i+2}) =
	(2r',2r'-2,2r'-4)$, where $r'=r-i+1$, and $(\lambda^{\bi_{k}}_i,\lambda^{\bi_{k}}_{i+1},\lambda^{\bi_{k}}_{i+2}) =
	(2r'-2,2r'-2,2r'-4)$. Here we can exclude the possibility that $r' = 1$ as then we would have $k = s$, so we assume that $r' \ge 2$.
	
	As $j_k > i$ we have that $j_k = j_{k+1} = i+1$ by KS-property~\ref{KS:2diff}.
	Then we get $(\lambda^{\bj_{k+1}}_i,\lambda^{\bj_{k+1}}_{i+1},\lambda^{\bj_{k+1}}_{i+2}) =
	(2r'-4,2r'-5,2r'-5)$, and thus $(j_k,j_{k+1},j_{k+2}) = (i+1,i+1,a)$ with $a \ge i+3$. From this 
	we can exclude the case $r'=2$, and thus assume that $r' \ge 3$.

	We consider the possibilities for $i_{k+1}$.  By using
	KS-property~\ref{KS:2diff} (at the index $i+1$) we have that $i_{k+1}$ is $i$, $i+1$ or $i+2$. Then:
	\begin{itemize}
		\item If $i_{k+1} = i+2$, then by KS-property~\ref{KS:2diff} we also have $i_{k+2} = i+2$. (We note that this case is not possible for $r'=3$.)
		
		\item If $i_{k+1} = i+1$, then we have $(\lambda^{\bi_{k+1}}_i,\lambda^{\bi_{k+1}}_{i+1},\lambda^{\bi_{k+1}}_{i+2})
		= (2r'-4,2r'-4,2r'-4)$ and see that $i_{k+2} \ge i+2$. 
		
		\item If $i_{k+1} = i$, then we have $(\lambda^{\bi_{k+1}}_i,\lambda^{\bi_{k+1}}_{i+1},\lambda^{\bi_{k+1}}_{i+2})
		= (2r'-3,2r'-3,2r'-4)$ and see that $i_{k+2} \ge i+2$. 
	\end{itemize}
	
	Summarising we get that  $(i_k,i_{k+1},i_{k+2})$ is one of $(i,i+2,i+2)$, $(i,i+1,b)$,
	where $b \ge i+2$, or $(i,i,c)$ where $c \ge i+2$.
	
	We now consider $h^{-1}(k)$, which is a subset of $\{k,k+1,\dots,s\}$ such
	that $\sum_{l \in h^{-1}(k)} i_l = i+1$.  Similarly, $h^{-1}(k+1)$ is a subset of $\{k,k+1,\dots,s\}$ such
	that $\sum_{l \in h^{-1}(k+1)} i_l = i+1$. We see from the above possibilities for $(i_k,i_{k+1},i_{k+2})$ 
	that this is not possible.
	
	Therefore, we have shown that $i_{k-1}<i$ is impossible, and we have our contradiction in this case.  
	
	\smallskip
	
	\noindent
	{\em Case $i_{k-1}=i$.}  Then by Observation~\ref{o1} we must have $i_{k-1} = i$ corresponding to a type 1 reduction, and  $i_k = i$ corresponding
	to a type 2 reduction. For the type 2 reduction to be possible we require that $\lambda^{\bi_{k-1}}_{i-1} > \lambda^{\bi_{k-1}}_{i}$. 
	We then have that $(\lambda^{\bi_{k-1}}_i,\lambda^{\bi_{k-1}}_{i+1},\lambda^{\bi_{k-1}}_{i+2}) =
	(2r'-2,2r'-2,2r'-4)$
	and  $(\lambda^{\bi_{k}}_i,\lambda^{\bi_{k}}_{i+1},\lambda^{\bi_{k}}_{i+2}) =
	(2r'-3,2r'-3,2r'-4)$, where $r' = r-i+1$. 
	
	Thus we have that $i_{k+1} \ge i+2$. 
	Since $\lambda^{\bj_{k-1}} = \lambda^{\bi_{k-1}}$ and $j_k > i_k$ we must have that  $(j_k,j_{k+1})$ is one
	of $(i+2,i+2)$ or $(i+1,b)$ where $b \ge i+2$.
	
	In the case $(j_k,j_{k+1}) = (i+1,b)$ we see that $h^{-1}(k)$ is a subset of $\{k,k+1,\dots, s\}$ such that 
	$\sum_{l \in h^{-1}(k)} i_l = i+1$. Since $i_{k+1} > i+1$ the only possible element of $h^{-1}(k)$ is $k$, but
	as $i_k = i < i+1$ we see that this is not possible.
	
	We show that the case $(j_k,j_{k+1}) = (i+2,i+2)$ is in fact also not possible. For this case we would have by Observation~\ref{o1} that 
	$(\lambda^{\bj_{k+1}}_i,\lambda^{\bj_{k+1}}_{i+1},\lambda^{\bj_{k+1}}_{i+2}) =
	(2r'-6,2r'-6,2r'-7)$, and also that $\lambda^{\bj_{k+1}}_{i-1} > 2r'-6$ and $\lambda^{\bj_{k+1}}_{i+3} = 2r'-7$. Thus we see that $j_l > i+2$ for all $l > k+1$. It follows that 
	$\lambda^{\bj}_{i-1} > \lambda^{\bj}_i = \lambda^{\bj}_{i+1} > \lambda^{\bj}_{i+2}$ and that $\lambda^{\bj}_i$ is even. This is a 
	contradiction to $\lambda^{\bj}$ being a rigid symplectic partition.
	
	Therefore, we have shown that $i_{k-1}=i$ is also impossible, and we have our contradiction in this case.
	
	\smallskip
	
	The contradictions in the two cases above imply that $i=j$ and therefore complete our inductive step. 
	
	\medskip
	
	(b) The proof for $G = \SO_{r^2}$ is very similar. We note that the case $r \le 4$ where $\lambda$ is one of $(7,5,3,1)$, $(5,3,1)$ or $(3,1)$ can be dealt with similarly 
	to Example~\ref{E:642}. So the assumption $r > 4$ can be made. The proof now proceeds in the same way just with $2r'+1$ in place of $2r'$. 
\end{proof}

For the proof
of this Theorem~\ref{T:nonconj} it is helpful to pick out the following useful property of the KS-algorithm.

\begin{ksprop} \label{KS:3diff}
	Let $\lambda$ be a symplectic partition or an orthogonal partition depending on which case we are considering.
	Let $\bi$ be a maximal admissible inc-sequence
	for $\lambda$, and let $i \in \Z_{\ge 1}$ with $i_1 \ge i$.
	Suppose that $\lambda_i > \lambda_{i+1} + 2$.    Then $i_1 = i$.
\end{ksprop}

\begin{proof}
	Suppose $i_1 > i$.  For $j \in \{1,\dots,s\}$ the only way we can have that
	$\lambda^{\bi_j}_i - \lambda^{\bi_j}_{i+1} < \lambda^{\bi_{j-1}}_i - \lambda^{\bi_{j-1}}_{i+1}$
	is if $i_j = i+1$ and this corresponds to a type 2 reduction in the KS-algorithm.  Then
	we have $\lambda^{\bi_j}_i - \lambda^{\bi_j}_{i+1} = \lambda^{\bi_{j-1}}_i - \lambda^{\bi_{j-1}}_{i+1} - 1$.
	By considering the parity of $\lambda^{\bi_j}_{i+1}$ there can be only
	one such type 2 reduction.  Thus we have that $\lambda^\bi_i \ge \lambda^\bi_{i+1} + 2$, which
	is contrary to $\lambda^\bi$ being a rigid partition.  Hence, we can conclude that $i_1 = i$.
\end{proof}

We are now ready to prove Theorem~\ref{T:nonconj} and note that
the idea of  proof of the theorem is to reduce to a case given in Lemma~\ref{L:specialpartitions}.

\begin{proof}[Proof of Theorem~\ref{T:nonconj}]
	We first cover the case $G = \SL_n$.  There is a unique rigid induction datum  $[\g_0,\O_0]$ for $\O$,
	as explained in \S\ref{ss:LSind}.  Thus trivially $\g_1 = \g_2$ and $\O_1 = \O_2$ in this case.
	
	\smallskip
	
	We move on to consider the case $G = \Sp_{2n}$ so $\g = \sp_{2n}$.
	Let $\lambda = (\lambda_1,...,\lambda_r)$ be the symplectic partition such that $\O =\O_\lambda$.
	
	By \cite[Lemma~8.4]{To} we can find $\bj = (j_1,\dots,j_t)$ such that the orbit
	$\{0\} \times \O_\lambda$ in $\sp_{\bj,2n}= \gl_{\bj} \times \g$ induces to an orbit $\O_{\tilde \lambda}$ in $\sp_{2|\bj|+2n}$ such that the symplectic partition $\tilde \lambda = (\tilde \lambda_1,...,\tilde \lambda_{\tilde r})$ of $2|\bj|+2n$ satisfies that $\tilde\lambda_i$ is even for all $i$ and
	$\tilde\lambda_i > \tilde \lambda_{i+1}$ for all $i = 1,...,\tilde r$.
	We have that $\gl_{\bj} \times \g_1$ and $\gl_{\bj} \times \g_2$ are Levi subalgebras of $\sp_{\bj,2n}$ and thus
	Levi subalgebras of $\sp_{2|\bj|+2n}$.  It follows that $[\gl_{\bj} \times \g_1,\{0\} \times \O_1]$ and
	$[\gl_{\bj} \times \g_2,\{0\} \times \O_2]$ are rigid induction data for $\O_{\tilde\lambda}$.
	
	We have  $\gl_{\bj} \times \g_1 \sub \gl_{\bj} \times \g_2$.
	Moreover $\g_1 = \g_2$ and $\O_1 = \O_2$ 
	if and only if $\gl_{\bj} \times \g_1 = \gl_{\bj} \times \g_2$ and $\{0\} \times \O_1 = \{0\} \times \O_2$.  By replacing $\lambda$ by $\tilde \lambda$, we can
	therefore assume that $\lambda$ satisfies that $\lambda_i$ is even for all $i$ and
	$\lambda_i > \lambda_{i+1}$ for all $i = 1,...,r$.
	
	Let $\bi' = (i'_1,...,i'_{s'})$ be the inc-sequence in which
	$i \in \{1,...,r\}$ occurs with multiplicity $(\lambda_i - \lambda_{i+1})/2 - 1$.
	Then we have that $\bi'$ is an admissible sequence for $\lambda$, for which reductions
	of the KS-algorithm are of type 1.  Further, we have that
	$\lambda^{\bi'} = (2r,2r-2,\dots,4,2)$ is a partition as in Lemma~\ref{L:specialpartitions}.
	
	Let $\bi$ be a maximal admissible sequence for $\lambda$.  Using KS-property~\ref{KS:order}
	we may first assume that $\bi$ is an inc-sequence without affecting $\lambda^\bi$ or $\sp_{\bi,2n_\bi}$ up to conjugacy.
	Then we observe that
	$\bi$ must start with at least $l := (\lambda_1 - \lambda_2)/2 - 1$ entries equal
	to $1$ by using KS-property~\ref{KS:3diff}.  We see that the maximal $k$ such that
	$i_k = 1$ is  $l$, $l+1$ or $l+2$, and then the next $(\lambda_2 - \lambda_3)/2 - 1$ entries in $\bi$ must be equal
	to $2$ by using KS-property~\ref{KS:3diff}.  Continuing this argument we see that
	for each $i$ there must be at least $(\lambda_i - \lambda_{i+1})/2 - 1$ entries equal to $i$
	in $\bi$.  In other words, $\bi'$ is a subsequence of $\bi$.
	Another application of KS-property~\ref{KS:order} allows us reorder $\bi$ to assume that $\bi_{s'} = \bi'$
	whilst not affecting $\sp_{\bi,2n_\bi}$ or $\lambda^\bi$.

	It now follows that there is a bijection from the rigid induction data for $\O_{\lambda} \sub \sp_{2n}$
	to the rigid induction data for $\O_{\lambda^{\bi'}} \sub \sp_{2n_{\bi'}}$.
	The bijection is given by $[\sp_{\bi,2n_\bi}, \{0\}^s \times \O_{\lambda^\bi}] \mapsto [\sp_{\bi'',2(n_{\bi'})_{\bi''}},\{0\}^{s''} \times \O_{(\lambda^{\bi'})^{\bi''}}]$,
	where $\bi'' = (i''_1,...,i''_{s''})$ is such that $\bi = (i'_1,...,i'_{s'},i''_1,...,i''_{s''})$; so that $(n_{\bi'})_{\bi''} = n_\bi$ and $(\lambda^{\bi'})^{\bi''} = \lambda^\bi$.
	
	We observe that $[\g_1,\O_1]$ and $[\g_2,\O_2]$ are rigid
	induction data for $\O_{\lambda^{\bi'}} \sub \sp_{\bi',n_{\bi'}}$.
	Thus by Lemma~\ref{L:specialpartitions} we obtain that $\g_1=\g_2$ and $\O_1 = \O_2$.
	This completes the proof for the symplectic case.
	
	\smallskip
	
	We end the proof by considering the case $G= \SO_N$.  Here we can proceed exactly as in the symplectic case.   The only minor changes
	required are to replace all occurrences of $\sp$ with $\so$, to take all entries of $\tilde \lambda$ to be odd then assume that all $\lambda_i$ are odd and distinct, and to say that $\lambda^{\bi'}$
	is equal to $(2r-1,2r-3,\dots,3,1)$.
\end{proof}

\begin{rmk} \label{R:nonconj}
	We note that the statement of Theorem~\ref{T:nonconj} holds more generally for any reductive algebraic group $G$ over $\C$.
	First, from the definition of Lusztig--Spaltenstein induction it is straightforward to see that the theorem
	reduces to the case where $G$ is simple and that it is independent of the isogeny class of $G$.  So we are just
	left to consider $G$ simple of exceptional type.
	
	The description of induced orbits and the Levi subalgebras from which they are induced for $G$ of exceptional type can be found in
	\cite[Tables~6--10]{EdG}.  The majority of induced orbits have a unique rigid induction datum, and so the statement
	of  Theorem~\ref{T:nonconj} holds trivially for these orbits. For all the other orbits a case by case analysis shows that
	there can never be containment of Levi subalgebras for distinct rigid induction data.  Thus the statement of Theorem~\ref{T:nonconj}
	does hold for $G$ of exceptional type.
	
	We also comment here that it would be interesting to find a general conceptual
	argument to prove this more general version of  Theorem~\ref{T:nonconj}.
\end{rmk}

\section{Finite \texorpdfstring{$W$}{W}-algebras and Losev--Premet ideals} \label{S:Walg}

For this section $G$ is any connected reductive algebraic group over $\C$.  We let $e \in \g$ be
a nilpotent element and let $\O = \O_e$.  We recall the required background on
the finite $W$-algebra $U(\g,e)$, move on to discuss Losev's parabolic induction
and then provide the setup required to state Theorem~\ref{T:surj}.  As a general reference for this material on
finite $W$-algebras we refer for example to \cite{LoICM}, and note that this covers all of the fundamental results
recapped below.

To define $U(\g,e)$ we extend $e$ to an $\sl_2$-triple $(e,h,f)$ in $\g$ and
consider the $\ad h$-grading $\g = \bigoplus_{i\in \Z} \g(i)$
of $\g$.  We write $\chi := (e, \cdot) \in \g^*$, where we recall that
$(\cdot\,,\cdot)$ is a nondegenerate invariant symmetric bilinear form on $\g$.
There is a symplectic form on $\g(-1)$ defined by $(x,y) \mapsto \chi([x,y])$, and we choose
a Lagrangian subspace $\l \sub \g(-1)$.  We let $\m := \l \oplus \bigoplus_{i < -1} \g(i)$ and
$\m_\chi := \{ x - \chi(x) \mid x\in \m\} \sub U(\g)$.  Let $Q_\chi := U(\g) / U(\g) \m_\chi$,
which is a (left) $U(\g)$-module called the {\em generalised Gelfand--Graev module} associated to  $\chi$.
The {\em finite $W$-algebra} corresponding to $\g$ and $e$ is defined as
$$
U(\g,e) := \End_{U(\g)}(Q_\chi)^\op.
$$
It is known that up to isomorphism $U(\g,e)$ does not depend on the choice of $\l$
or on the choice of $\sl_2$-triple, and in fact only depends on the
$G$-orbit of $e$.
\begin{rmk}\label{r: dir sum}
	If $G_1$ and $G_2$ are both connected reductive algebraic groups over $\C$ with Lie algebras $\g_1$ and $\g_2$, then $G_1\times G_2$ is a connected reductive algebraic group over $\C$ with Lie algebra $\g_1\oplus \g_2$. Letting $e_1$ be a nilpotent element of $\g_1$ and $e_2$ a nilpotent element of $\g_2$, we get that $e_1+e_2$ is a nilpotent element of $\g_1\oplus\g_2$. From the construction above, and using that $U(\g_1\oplus\g_2)\cong U(\g_1)\otimes U(\g_2)$, it is straightforward to see that $U(\g_1\oplus\g_2,e_1+e_2)\cong U(\g_1,e_1)\otimes U(\g_2,e_2)$.
\end{rmk}
By using an alternative definition of $U(\g,e)$, one can get an action of $(G^e)^h$ on $U(\g,e)$ (see \cite[\S 1.2]{GG} and \cite[\S 2.2]{PrJo}). This induces an action of $(G^e)^h$ on two-sided ideals of $U(\g,e)$; in fact, since $((G^e)^h)^\circ$ preserves any two-sided ideal (see \cite[Proof of Corollary 2.1]{PrJo}), this descends to an action of the component group of $((G^e)^h)^\circ$ on the set of two-sided ideals of $U(\g,e)$. A $U(\g,e)$-module $M$ may also be twisted by $g\in (G^e)^h$ to a module ${}^gM$, which coincides with $M$ as a vector space but on which an element $u\in U(\g,e)$ acts as $g^{-1}\cdot u$ does on $M$. Note that $\Ann_{U(\g,e)}({}^gM)=g\cdot \Ann_{U(\g,e)}(M)$.

The generalised Gelfand-Graev module $Q_\chi$ is a left $U(\g)$-module and a right $U(\g,e)$-module.
Let $\Wh_\chi(U(\g))$ denote the full subcategory of $U(\g)$-modules on which $\m_\chi$
acts locally nilpotently.
The functor
\begin{equation} \label{e:Skryabinsequiv}
	Q_\chi\otimes_{U(\g,e)} (-) : U(\g,e)\lmod \isoto \Wh_\chi(U(\g))
\end{equation}
is an equivalence of categories known as {\em Skryabin's equivalence}.
This leads to a close relationship between finite dimensional simple $U(\g,e)$-modules and the set $\Prim_\O U(\g)$, as 
established within \cite[Theorem~1.2.2]{LoQS} and which we now describe.

For $M$ a finite dimensional simple $U(\g,e)$-module,
$\Ann_{U(\g)}(Q_\chi\otimes_{U(\g,e) }M)$ lies in $\Prim_\O U(\g)$; conversely, any $I \in \Prim_\O(U(\g))$ is of the
form $\Ann_{U(\g)}(Q_\chi\otimes_{U(\g,e) }M)$ for some finite dimensional simple $U(\g,e)$-module $M$.
Moreover, by \cite[Conjecture 1.2.1]{LoDuke}, two finite dimensional simple $U(\g,e)$-modules $M$ and $N$ satisfy $\Ann_{U(\g)}(Q_\chi\otimes_{U(\g,e) }M) = \Ann_{U(\g)}(Q_\chi\otimes_{U(\g,e) }N)$
if and only if $M$ is isomorphic to a twist of $N$ by an action
of the component group
of $(G^e)^h$ on $U(\g,e)$.
Lastly, from \cite[\textsection 1.2]{PT} and the references therein we recall that for a $1$-dimensional $U(\g,e)$-module $M$ it is known that $\Ann_{U(\g)}(Q_\chi\otimes_{U(\g,e) }M)$ is a completely
prime primitive ideal, i.e. it lies in $\Prim_\O^1(U(\g))$.

The $1$-dimensional representations of $U(\g,e)$
correspond to the points of
the affine variety
\begin{equation} \label{e:Ege}
	\cE(\g,e) := \Spec_{\mathrm{max}}(U(\g,e)^{\ab}),
\end{equation}
where $U(\g,e)^{\ab} := U(\g,e) / ([a,b] \mid a,b \in U(\g,e))$. Since there is a bijection between maximal ideals of $U(\g,e)^\ab$ and $1$-dimensional representations of $U(\g,e)$, the action of the component group of $(G^e)^h$ on the former set induces one on the latter; this is compatible with the twisting operation discussed earlier. It follows from the discussion above that there is a map
\begin{equation} \label{e:SonE}
	\sS=\sS(\g,e) :=  \Ann_{U(\g)}(Q_\chi\otimes_{U(\g,e)} (-)) :\cE(\g,e) \to \Prim_\O^1(U(\g)),
\end{equation}
whose fibres are orbits under the action of the component group of $(G^e)^h$.

\begin{dfn}
	An ideal $I$ in $U(\g)$ is called a {\em Losev-Premet ideal} if it lies in the image of the map $\sS$.
\end{dfn}

To discuss parabolic induction of Losev-Premet ideals, we first need to refine our earlier discussion of rigid induction data. Recall (see \S\ref{ss:LSind}) that $\cR(\g,\O)$ denotes the finite set of rigid induction data for $\O$. Let us now enumerate this set as $\cR(\g,\O) = \{[\g_i, \O_i] \mid i=0,...,m\}$. From now on, we consider each rigid induction datum $[\g_i,\O_i]$ to come equipped with a fixed representative $(\g_i,\O_i)$ such that $\g_i$ is a standard Levi subalgebra, a fixed element $e_i\in\O_i$, and a parabolic subalgebra $\p_i$ of $\g$ of which $\g_i$ is a Levi factor. 

Now, fix $i\in\{0,\ldots,m\}$. In \cite[Theorem 1.2.1]{LoPI} Losev established a dimension preserving parabolic induction functor from finite dimensional
$U(\g_i,e_i)$-modules to finite dimensional $U(\g,e)$-modules.
At the level of $1$-dimensional modules this induces
a finite morphism
\begin{equation} \label{e:Losevsmorphism}
	\Phi_i: \cE(\g_i,e_i) \to \cE(\g,e),
\end{equation}
see \cite[Theorem~6.5.2]{LoPI}. The construction of this morphism depends on a choice of a parabolic subalgebra; in our set-up, we always use $\p_i$. This map $\Phi_i$ fits into the commutative diagram
\begin{equation} \label{e:inductiondiagram}
	\begin{array}{c}
		\xymatrix{
			\cE(\g_i,e_i) \ar@{->}[rrr]^{\sS_i} \ar@{->}[d]^{\Phi_i} & & & \Prim^1_{\O_i} U(\g_i) \ar@{->}[d]^{\Ind_{\p_i}^\g}\\
			\cE(\g,e) \ar@{->}[rrr]^{\sS} & & & \Prim^1_\O U(\g), }
	\end{array}
\end{equation}
where $\sS_i:=\sS(\g_i,e_i)$. Note that the commutativity of this diagram particularly implies that Losev-Premet ideals parabolically induce to Losev-Premet ideals (induction of ideals is recapped in \S\ref{ss:primitive}).

Combining the diagrams \eqref{e:inductiondiagram} for each $i=0,\ldots,m$, we obtain the commutative diagram
\begin{eqnarray}\label{e:diag}
	\begin{array}{c}
		\xymatrix{
			\bigsqcup_{i=0}^m \cE(\g_i,e_i) \ar@{->}[rrr]^{\sS_{\rig}} \ar@{->}[d]^{\Phi} & & &  \bigsqcup_{i=0}^{m}\Prim_{\O_i}^1 U(\g_i) \ar@{->}[d]^{\Ind^{\g}}\\
			\cE(\g,e) \ar@{->}[rrr]^{\sS} & & & \Prim_{\O}^1(\g,e).
		}
	\end{array}
\end{eqnarray}

Here we write $\Phi:=\bigsqcup_{i=0}^m \Phi_i$, $\sS_{\rig} := \bigsqcup_{i=0}^m \sS_{i}$, and $\Ind^{\g} := \bigsqcup_{i=0}^m \Ind_{\p_i}^{\g}$. Note that $\Phi$ in diagram \eqref{e:diag} is a finite morphism.
The following theorem is the main technical result of the paper, and immediately implies Theorem~\ref{T:maintheorem}.

\begin{theorem} \label{T:surj}
	If $G$ is a simple algebraic group over $\C$ of classical type, then $\Phi$ in \eqref{e:diag} is surjective.
\end{theorem}

We prove Theorem~\ref{T:surj} in full generality in \S\ref{S:FWACLA}, but to introduce some of the ideas of the proof we consider first the case of $m=0$, which is already known due to \cite[Theorem~5]{PT}. Note that this case corresponds to $\O$ having a unique rigid induction datum; by \S\ref{ss:sheets}, this is the same as $\O$ lying on a unique sheet. The argument follows the proof of \cite[Theorem~5]{PT}.

\begin{prop} \label{p:ASurj}
	Let $G$ be a simple algebraic group over $\C$ of classical type and let $\O$ be a nilpotent orbit in $\g$ with unique rigid induction datum $[\g_0,\O_0]$. Fix $e\in\O$. Then $U(\g,e)^{\ab}$ is a polynomial ring in $\dim \z(\g_0)$ variables and the map $\Phi_0$ in \eqref{e:Losevsmorphism} is surjective.
\end{prop}

\begin{proof}
	
	By \cite[Theorem~1]{PT} and \cite[Theorem~4(ii)]{PT}, $U(\g,e)^{\ab}$ is a polynomial algebra in $c(e)$ variables, where $c(e)=\dim \bar{\cD}^{\reg}(\g_0, \O_0)-\dim \O$. This is equal to $\dim\z(\g_0)$ by the discussion at the end of \S\ref{ss:sheets} (citing \cite[Satz 4.5]{BK}), and thus the first claim is proved.
	
	On the other hand it follows from Remark~\ref{r: dir sum} that
	$U(\g_0, e_0) \cong S(\z(\g_0)) \otimes U([\g_0, \g_0], e_0)$ and thus $U(\g, e_0)^{\ab} \cong S(\z(\g_0)) \otimes U([\g_0, \g_0], e_0)^{\ab}$. Applying \cite[Theorem~1]{PT}
	again, we see that $U([\g_0, \g_0], e_0)^{\ab}\cong \C$ and thus that $U(\g, e_0)^{\ab} \cong S(\z(\g_0))$ is a polynomial algebra in $\dim\z(\g_0)$ variables.
	
	We therefore conclude that $\Phi_0:\cE(\g_0,e_0)\to\cE(\g,e)$ is a finite morphism between affine spaces of the same dimension. Since finite morphisms are closed,
	$\Phi$ is surjective.
\end{proof}

As explained in \S\ref{ss:LSind},
a nilpotent orbit $\O$ in $\sl_n$ has a unique rigid induction datum and so lies in a unique sheet. Proposition~\ref{p:ASurj} is thus sufficient to prove Theorem~\ref{T:surj} in type {\sf A} and we therefore get the following result.

\begin{cor} \label{C:typeA}
	If $G$ is a simple algebraic group over $\C$ of type {\sf A} then $\Phi$ in \eqref{e:diag} is surjective.
\end{cor}

For types {\sf B}, {\sf C} and {\sf D} we may have $m>0$ and so we need to work a bit harder; we do so in the next section.

\section{Proof of Theorem~\ref{T:surj}}
\label{S:FWACLA}
Thanks to Corollary~\ref{C:typeA}, what remains is to prove Theorem~\ref{T:surj} when $G$ is $\Sp_{2n}$ or $\SO_N$. For brevity we tackle only $G=\Sp_{2n}$ in this section, but the argument works almost identically for $G=\SO_N$.

\subsection{Preliminary observations}\label{ss:PrelObs}

Before getting into the substance of the proof of Theorem~\ref{T:surj} it will be useful to make a handful of observations about one-dimensional representations and central characters in the case of $\g=\sp_{2n}$ and its Levi subalgebras (analogous results hold for $\so_N$). We do so here.

As discussed in \S\ref{ss:paralevi}, the standard Levi subalgebras of $\sp_{2n}$ are all of the form $$\sp_{\bi,2n_\bi}=\gl_{i_1}\times\cdots\times\gl_{i_s}\times \sp_{2n_\bi}$$
for $\bi=(i_1,\ldots,i_s)$ an inc-sequence with $|\bi| \le n$, and with $n_\bi := n-|\bi|$. In particular, each $\sp_{\bi,2n_\bi}$ contains $\t=\Lie(T)$, where $T$ is the maximal torus consisting of diagonal matrices in $\Sp_{2n}$. Furthermore, the rigid induction data are of the form $[\sp_{\bi,2n_\bi}, \{0\} \times \O_0]$, where $\O_0$ is a rigid nilpotent orbit in $\sp_{2n_\bi}$.

For such a rigid induction datum let $e_0 \in \O_0$; we sometimes abuse notation to identify $e_0\in\O_0$ with $(0,e_0)\in\{0\}\times \O_0$, thus viewing it as a nilpotent element in $\sp_{\bi,2n_{\bi}}$.  As algebras, we have by Remark~\ref{r: dir sum}
\begin{eqnarray}
	\label{e:FWAdecomp}
	U(\sp_{\bi,2n_{\bi}}, e_0) \cong U(\gl_{i_1}) \otimes \cdots \otimes U(\gl_{i_s}) \otimes U(\sp_{2n_\bi}, e_0)
\end{eqnarray}
and thus
\begin{eqnarray}
	\label{e:FWAdecompAb}
	U(\sp_{\bi,2n_{\bi}}, e_0)^{\ab} \cong S(\z(\gl_{i_1})) \otimes \cdots \otimes S(\z(\gl_{i_s})) \otimes U(\sp_{2n_\bi}, e_0)^{\ab}.
\end{eqnarray}
Noting that $\z(\sp_{\bi,n_{\bi}})=\z(\gl_{i_1})\times\cdots\times \z(\gl_{i_s})\times\{0\}$ and that \cite[Theorem 1]{PT} implies that $U(\sp_{2n_\bi}, e_0)^{\ab}\cong\C$ (since $e_0$ is rigid in the simple Lie algebra $\sp_{2n_\bi}$), we can simplify this as
\begin{eqnarray}
	\label{e:FWAdecompAb2}
	U(\sp_{\bi,2n_{\bi}}, e_0)^{\ab} \cong S(\z(\sp_{\bi,2n_\bi})).
\end{eqnarray}

In particular, 	$U(\sp_{\bi,2n_{\bi}}, e_0)^{\ab}$ is a polynomial algebra in $\dim\z(\sp_{\bi,2n_\bi})=s$ variables and $\cE(\sp_{\bi,2n_{\bi}}, e_0)$ is $s$-dimensional affine space.

It will be helpful for us to give a more explicit description of $\cE(\sp_{\bi,2n_{\bi}}, e_0)$; we do so by turning our attention to the central characters corresponding to one-dimensional $U(\sp_{\bi,2n_\bi},e_0)$-modules.  We have seen that there is a unique $\C$-algebra homomorphism $\eta=\eta(\bi,e_0) : U(\sp_{2n_\bi},e_0) \to \C$, and we denote the corresponding one-dimensional
$U(\sp_{2n_\bi},e_0)$-module by $\C_\eta$.
Thanks to \eqref{e:FWAdecompAb}, for each $\lambda\in \z(\sp_{\bi,2n_\bi})^*$ there is a unique one-dimensional
representation $\C_\eta^\lambda$ of $U(\sp_{\bi,2n_\bi},e_0)$ on which $U(\gl_{i_1}) \otimes \cdots \otimes U(\gl_{i_s})$
acts via $\lambda$ and $U(\sp_{2n_\bi}, e_0)$ acts via $\eta$.
Therefore,
\begin{equation} \label{e:Ele0}
	\cE(\sp_{\bi,2n_\bi}, e_0) = \{ \C_\eta^\lambda \mid \lambda \in \z(\sp_{\bi,2n_\bi})^*\}.
\end{equation}

The map $\sS_0:=\sS(\sp_{\bi,2n_\bi}, e_0)$ then defines for each $\lambda \in \z(\sp_{\bi,2n_\bi})^*$ a Losev-Premet ideal $$I_\eta^\lambda := \sS_0(\C_\eta^\lambda)\subseteq U(\sp_{\bi,2n_{\bi}}).$$ Using Duflo's theorem, there exists $\mu=\mu(\bi,e_0) \in \t^*$ such that $I_\eta^0 = \Ann_{U(\sp_{\bi,2n_\bi})} L(\mu)$.
It is
not hard to see that
\begin{eqnarray} \label{e:relabellingistwisting}
	I_\eta^\lambda = \Ann_{U(\sp_{\bi,2n_\bi})} (L(\mu) \otimes \C_\lambda),
\end{eqnarray}	
where $\C_\lambda$ is the one-dimensional $U(\sp_{\bi,2n_\bi})$-module upon which $U(\sp_{2n_\bi})$ acts trivially and
$U(\gl_{i_1}) \otimes \cdots \otimes U(\gl_{i_s})$ acts via $\lambda$. Let $W_\bi$ be the Weyl group corresponding to $(\Sp_{\bi,2n_\bi},T)$ and let $\ch_\bi$ be the corresponding central character map from \S\ref{ss:primitive}.  We observe that $L(\mu) \otimes \C_\lambda$ is a highest weight $U(\sp_{\bi,2n_\bi})$-module with highest weight $\mu+\lambda$; as discussed in \S\ref{ss:primitive}, this implies that  $\ch_\bi(I_\eta^\lambda)=W_\bi \bullet (\mu+\lambda)$. Thus, using  \eqref{e:centralchars}, we have:

\begin{lemma} \label{L:lemmawithsomename}
	$\ch(\Ind_{\p_\bi}^{\sp_{2n}}(I_\eta^\lambda)) = W\bullet (\mu+\lambda).$
\end{lemma}

\begin{rmk}
	When $\g$ is any simple Lie algebra over $\C$ with rigid induction datum $[\g_0,\O_0]$, the decomposition $\g_0=\z(\g_0)\times [\g_0,\g_0]$ can be used to obtain analogues of the above results so long as $U([\g_0,\g_0],e_0)^{\ab}\cong \C$. Reducing to the case where $[\g_0,\g_0]$ is simple, this holds in all classical cases by \cite{PT} and in all but 6 exceptional types by \cite{PrMF} -- using Bala-Carter notation, these 6 cases correspond to the nilpotent orbits $\tilde{A}_1$ in $G_2$, $\tilde{A}_2+A_1$ in $F_4$, $(A_3+A_1)'$ in $E_7$, and $A_3+A_1$, $A_5+A_1$ and $D_5(a_1)+A_2$ in $E_8$. In each of these cases, $U([\g_0,\g_0],e_0)^{\ab}\cong \C\times \C$ by \cite[Theorem A]{PS} and there are thus two one-dimensional representations $\eta_1,\eta_2:U([\g_0,\g_0],e_0)\to\C$. One then gets that $$\cE(\g_0, e_0) = \{ \C_{\eta_i}^\lambda \mid i=1,2\mbox{ and }\lambda \in \z(\g_0)^*\}.$$  All the other observations work as for $\sp_{2n}$, except that $\mu$ will depend on which $\eta_i$ is being considered.
\end{rmk}

The final result we need is the following simple lemma. Recalling that $W$ acts on $\t^*$ via the dot-action, let us write $V/(W,\bullet)$ for the image of any subset $V$ of $\t^*$ under the surjection $\t^*\twoheadrightarrow \t^*/(W,\bullet)$.

\begin{lemma} \label{GroupAct}
	Let $\mu_0,\ldots,\mu_m\in \t^{*}$ and let $\z_0,\ldots,\z_m \sub \t^{*}$ be vector subspaces. If
	$(\mu_0+\z_0)/(W,\bullet)$ is contained in $\bigcup_{i=1}^m(\mu_i+\z_i)/(W,\bullet)$ then
	there exists $1\leq i_0\leq m$ and $w_0\in W$ such that $\z_0\sub w_0(\z_{i_0})$.
\end{lemma}

\begin{proof}
	The inclusion $(\mu_0+\z_0)/(W,\bullet)\sub \bigcup_{i=1}^{m} (\mu_i+\z_i)/(W,\bullet)$ implies that
	$$
	\mu_0+\z_0\sub \bigcup_{i=1}^{m} \bigcup_{w\in W} w\bullet(\mu_i+\z_i).
	$$
	The right hand side is an affine variety with irreducible components of the form
	$w\bullet(\mu_i+\z_i)$, where $w\in W$ and $i\in \{1,...,m\}$. Since $\mu_0 + \z_0$
	is an irreducible subvariety there exist $1\leq i_0\leq m$ and $w_0\in W$ such that
	$\mu_0+\z_0\sub w_0\bullet(\mu_{i_0}+\z_{i_0}).$ Now, $\z_0\sub w_0\bullet\mu_{i_0}-\mu_0+w_0(\z_{i_0})$
	and so $0\in\z_0$ implies $w_0\bullet\mu_{i_0}-\mu_0+w_0(\z_{i_0})=w_0(\z_{i_0})$. The result follows.
\end{proof}

\subsection{Main proof}

We are now in a position to proceed to the main steps in the proof of Theorem~\ref{T:surj}. From now on, we recall the refined set-up of rigid induction data from \S\ref{S:Walg} and our notation is consistent with that section. In particular, we label our rigid induction data by $[\g_i,\O_i]$, $i=0,\ldots,m$, rather than by $[\sp_{\bi,2n_{\bi}},\{0\}\times\O_0]$ with $\bi$ a maximal admissible inc-sequence and $\O_0$ a rigid nilpotent orbit in $\sp_{2n_\bi}$. The reader should nonetheless recall that our chosen representative of each  $[\g_i,\O_i]$ is selected to be of the form $(\sp_{\bi,2n_{\bi}},\{0\}\times\O_0)$ so that we may use the results of \S\ref{ss:PrelObs}.

\begin{prop} \label{P:Int}
	For $i =0,...,m$, we have $\Phi(\cE(\g_i, e_i)) \nsubseteq \bigcup_{j\ne i} \Phi(\cE(\g_j, e_j))$.
\end{prop}

\begin{proof}
	
	For each $i$, denote by $W_i$ the Weyl group associated to $(G_i,T)$. Proceeding as in \S\ref{ss:PrelObs}, we find $\eta_i\in\t^*$ and write $\cE(\g_i,e_i)=\{\C_{\eta_i}^\lambda\mid \lambda \in \z(\g_i)^*\}$. We pick $\mu_i\in \t^*$ as in \S\ref{ss:PrelObs} so that $\ch_i(I_{\eta_i}^0)=W_i\bullet\mu_i$ (where $\ch_i$ denotes the obvious central character map); as before, $\ch_i(I_{\eta_i}^\lambda)=W_i\bullet(\mu_i+\lambda)$ for $\lambda\in \z(\g_i)^*$.
	
	Suppose now for a contradiction that $\Phi(\cE(\g_i, e_i)) \sub \bigcup_{j\ne i} \Phi(\cE(\g_j, e_j)$.
	Since Diagram \eqref{e:diag} commutes we deduce that
	$$\ch \Ind_{\p_i}^{\g}\sS_i (\cE(\g_i,e_i))\subseteq \bigcup_{j\neq i} \ch \Ind_{\p_j}^{\g}\sS_j (\cE(\g_j,e_j))\subseteq \t^*/(W,\bullet).$$
	
	Using the notation $\z_k := \Ann_{\t^*} ([\g_k, \g_k])$ (which we easily identify with $\z(\g_k)^*$) and using
	Lemma~\ref{L:lemmawithsomename} we deduce that
	$$
	(\mu_i + \z_i)/(W,\bullet) \sub \bigcup_{j\ne i} (\mu_j + \z_j)/(W,\bullet).
	$$
	By Lemma~\ref{GroupAct} we deduce that $\z_i$ is $W$-conjugate
	to a subspace of $\z_j$ for some $j\neq i$. This is equivalent to saying that $\z(\g_i)$ is $W$-conjugate
	to a subspace of $\z(\g_j)$ for some $j\neq i$ which, in turn, implies that $\g_j$ is $G$-conjugate to a subset
	of $\g_i$. This contradicts Theorem~\ref{T:nonconj}, and the contradiction completes the proof.	
\end{proof}

Recalling from \S\ref{ss:sheets} that sheets of $\g$ containing $\O$ are in bijection with the rigid induction data $[\g_i,\O_i]$, $i=0,\ldots,m$, let us denote by $\cS_i$ the sheet corresponding to $[\g_i,\O_i]$. As in \S\ref{ss:sheets}, we can write $\cS_i=\bar{\cD}^{\reg}(\g_i, \O_i)$.

Now, recall the $\sl_2$-triple $(e,h,f)$ from the definition of the finite $W$-algebra $U(\g,e)$. The Katsylo variety is defined as
\begin{equation*}
	e + X := (e + \g^f) \cap \bigcup_{i=0}^m \cS_i.
\end{equation*}
Key to the proof of Theorem~\ref{T:surj} is the following result on the irreducible components of $e+X$; given a variety $Y$, we denote by $\Comp Y$ the set of irreducible components of $Y$.

\begin{prop} \label{P:dimpresbij}
	There are dimension preserving bijections
	$$
	\Comp \cE(\g,e) \isoto \Comp(e+X)
	\isoto \Comp \bigsqcup_{i=0}^m \cE(\g_i, e_i).
	$$
\end{prop}

\begin{proof}
	The first bijection is given in \cite[Theorem~1.1]{To}, see Remark~\ref{R:maphistory}.
	
	To establish the second bijection we first observe the following three facts about $e+X$ which allow us to describe $\Comp(e+X)$:
	\begin{enumerate}
		\item[(i)] The intersections $(e + \g^f) \cap \cS_i$ are all irreducible, thanks to
		\cite[Theorem~6.2]{Im}.
		\item[(ii)] There are no inclusions between the varieties $\{(e + \g^f) \cap \cS_i \mid i=0,...,m\}$;
		this can be immediately deduced from the fact that $G\cdot ((e + \g^f) \cap \cS_i) = \cS_i$ by \cite[Lemma 2.8(i)]{Im} and that, by definition, no sheet
		can admit an inclusion into another.
		\item[(iii)] The dimension of $(e + \g^f) \cap \cS_i$ is $\dim \z(\g_i)$ (see \cite[Lemma~2.8(i)]{Im}, for example).
	\end{enumerate}
	The result then follows from the observation in \eqref{e:FWAdecompAb2} that, for $G=\Sp_{2n}$, each $\cE(\g_i,e_i)$ is an affine space of dimension $\dim \z(\g_i)$.
\end{proof}
\begin{rmk}
	\label{R:maphistory}
	In \cite[Theorem~1.2]{PrCQ} Premet used reduction modulo $p$ to establish the existence of a map $\Comp \cE(\g,e) \to \Comp(e+X)$ which restricts to a dimension preserving bijection on a subset of the domain. The fact that this map is a bijection for $\g$ classical was the main result of \cite{To}. We summarise the proof in {\it op. cit.} for the interested reader. 
	
	In part I of {\it op. cit.}, Dirac reduction was applied to obtain a presentation of the Poisson algebra $\C[e+\g^f]$ for distinguished nilpotent elements in types {\sf B}, {\sf C}, {\sf D}, leveraging the Yangian presentation in type {\sf A} given by Brundan--Kleshchev.
	
	Using information on the structure of sheets of classical Lie algebras and the Poisson presentation, it was then shown (Theorem 8.8 of {\it op. cit.}) that the abelian quotient $\C[e+\g^f]^{\ab}$ is reduced. This statement was proven for all nilpotent elements, using the distinguished case, by an analytic argument (Theorem 8.9 of {\it op. cit.}). Finally, a deformation argument was used to show that $\cE(\g,e)$ cannot have more irreducible components than $\cE(\g,e)$.
\end{rmk}

\begin{proof}[Proof of Theorem~\ref{T:surj}] 
	
	When $G$ is of type ${\sf A}$ this is Corollary~\ref{C:typeA}. For $G$ of type ${\sf B}$, ${\sf C}$ or ${\sf D}$ we reduce to $G=\Sp_{2n}$ or $\SO_N$ and use Proposition~\ref{P:Int} and Proposition~\ref{P:dimpresbij} to argue as follows. 
	
	We know that $\Phi$ as in \eqref{e:diag} is a finite morphism. Thus, its image
	$\Phi \left(\bigsqcup_i \cE(\g_i, e_i)\right) \sub \cE(\g, e)$ is closed and for any irreducible
	component $Y$ of $\bigsqcup_i \cE(\g_i, e_i)$ we have $\dim \Phi(Y) = \dim Y$.
	Thanks to Proposition~\ref{P:Int} and Proposition~\ref{P:dimpresbij} we see that
	the multiset of dimensions of irreducible components of
	$\Phi \left(\bigsqcup_i \cE(\g_i, e_i)\right)$ is equal to the multiset of dimensions of the irreducible components of $\cE(\g, e)$. This
	implies that $\Phi$ is surjective, completing the proof of Theorem~\ref{T:surj}.
\end{proof}

\begin{proof}[Proof of Theorem~\ref{T:maintheorem}] 
	This follows directly from Theorem~\ref{T:surj} using the maps $\sS$ and $\sS_{\rig}$ (which are surjective onto the sets of Losev-Premet ideals with the appropriate associated varieties) and the commutativity of Diagram~\eqref{e:diag}.
\end{proof}

\end{document}